\documentclass[final,3p,times]{elsarticle}

\usepackage{amsmath}
\usepackage{amsfonts}
\usepackage{amssymb}
\usepackage[english]{babel}
\usepackage{xcolor}
\usepackage{graphicx}
\usepackage{caption}

\usepackage{tikz}
\usepackage{pgfplots}
\usepackage{subfig}

\usepackage{lmodern}
\usepackage{amsthm}
\usepackage{mathrsfs}
\usepackage{microtype}
\usepackage{mathscinet}
\usepackage{enumitem}
\usepackage{hyperref}
\hypersetup{hidelinks}

\theoremstyle{plain}
\newtheorem{thm}{Theorem}[section]

    \theoremstyle{definition}
    
    \theoremstyle{plain}
    \newtheorem{prop}[thm]{Proposition}
    \newtheorem{lemma}[thm]{Lemma}
    \theoremstyle{definition}
    
    \theoremstyle{plain}
    \newtheorem{cor}[thm]{Corollary}
    \theoremstyle{definition}
    \newtheorem{rmk}[thm]{Remark}

\DeclareMathOperator\supp{supp}

\DeclareMathOperator\meas{meas}
\DeclareMathOperator\Fuj{Fuj}

\DeclareMathOperator\loc{loc}
\DeclareMathOperator\NE{NE}

\begin{document}

\begin{frontmatter}



\title{ Fujita versus Strauss - a never ending story }


%
%
%
%

\author[frei]{Alessandro Palmieri
\ead{alessandro.palmieri.math@gmail.com}}
\author[frei]{Michael Reissig
}
\ead{reissig@math.tu-freiberg.de}


\address[frei]{Institute of Applied Analysis, Faculty for Mathematics and Computer Science, Technical University Bergakademie Freiberg, Pr\"{u}ferstra{\ss}e 9, 09596, Freiberg, Germany}


\begin{abstract}

In this paper,  we obtain a blow-up result for solutions to a semi-linear wave equation with scale-invariant dissipation and mass and power non-linearity, in the case in which the model has a \textquotedblleft wave like $\!$\textquotedblright $\,\!$ behavior.

In order to achieve this goal, we perform a change of variables that transforms our starting equation in a strictly hyperbolic semi-linear wave equation with time-dependent speed of propagation.

 Then, we  apply Kato's lemma to find a blow-up result for solutions to the transformed equation under some support and sign assumptions on the initial data. A special emphasis is placed on the limit case, that is, when the exponent $p$ is exactly equal to the upper bound of the range of admissible values of $p$ for which this blow-up result is valid. In this critical case an explicit integral representation formula for solutions of the corresponding linear Cauchy problem in 1d is derived.

 Finally, carrying out the inverse change of variables we get a non-existence result for global (in time) solutions to the original model.

\end{abstract}


\end{frontmatter}


\section{Introduction}\label{Intro}

The goal of the present paper is to prove a non-existence result for global (in time) solutions of the Cauchy problem for semi-linear wave equation with scale-invariant dissipation and mass and power non-linearity, i.e., for solutions to the following model:
\begin{align}\label{CP semi scale inv}
\begin{cases}
v_{\tau\tau}-\Delta_y v +\dfrac{\mu_1}{1+\tau} v_\tau+\dfrac{\mu_2^2}{(1+\tau)^2} v=|v|^p, \qquad \tau>0, \; y\in \mathbb{R}^n, \\  v(0,y)=v_0(y), \quad y\in \mathbb{R}^n,\\  v_\tau(0,y)=v_1(y),\quad y\in \mathbb{R}^n,
\end{cases}
\end{align} assuming in some sense that the damping and the mass terms make the equation \emph{hyperbolic-like} from the point of view of the critical exponent diving the set of admissible exponents into one set which allows to prove a blow-up behavior for global (in time) solutions and a second set which allows to prove a (global) in time result of at least small data Sobolev solutions.

The previous model is called \emph{scale-invariant}, since the corresponding linear model is invariant under the so-called \emph{hyperbolic scaling}
\begin{align*}
\widetilde{v}(\tau,y)=v(\lambda (1+\tau)-1,\lambda y), \quad \lambda>0.
\end{align*}

Let us formulate analytically, in terms of $\mu_1$ and $\mu_2^2$, the assumption that we require for our model in this paper.

If we define
\begin{align*}
\delta:=(\mu_1-1)^2-4 \mu_2^2,
\end{align*} then our assumption for these coefficients is
\begin{align}\label{condition on delta}
 \delta \in (0,1] .
\end{align}

As it is explained in \cite{NunPalRei16}, the quantity $\delta$ describes in some sense the interplay between the damping and the mass term in \eqref{CP semi scale inv} and in the corresponding linear problem. In other words, the qualitative properties of the solutions to \eqref{CP semi scale inv} are different for different ranges of $\delta$ (see for example \cite{NunPalRei16,PalRei17,Pal17}).

Considering the transformation
\begin{align}\label{transformation u,v}
u(t,x)= (1+\tau)^{\frac{\mu_1-1}{2}+\frac{\sqrt{\delta}}{2}}v(\tau,y) \qquad \tau=(1+t)^{\ell+1}-1, \;\; y=(1+\ell)x,
\end{align} where
\begin{align}\label{definition ell and k}
\ell=\tfrac{1-\sqrt{\delta}}{\sqrt{\delta}}, \quad k= \tfrac{1-\mu_1-\sqrt{\delta}}{2\sqrt{\delta}}(p-1)+\tfrac{2(1-\sqrt{\delta})}{\sqrt{\delta}},
\end{align}

then we find that $u$ solves the following Cauchy problem:

\begin{align}\label{CP semi Tric eq}
\begin{cases}
u_{tt}-(1+t)^{2\ell}\Delta_x u= (\ell+1)^2(1+t)^{k}|u|^p,\quad t>0,\,\, x\in \mathbb{R}^n, \\  u(0,x)=u_0(x), \quad x\in \mathbb{R}^n,\\  u_t(0,x)=u_1(x),\quad x\in \mathbb{R}^n,
\end{cases}
\end{align} for suitable $u_0,u_1$.

In particular, we see that condition \eqref{condition on delta} allows the choice of a nonnegative $\ell$.  

Therefore, we consider the Cauchy problem \eqref{CP semi Tric eq} for general $\ell \geq 0$, $k>-2$  and nonnegative compactly supported data $u_0$ and $u_1$. We will clarify the condition on $k$ after the statements of the main results in Section \ref{Section Main Theorems}.

 Since, we can derive a non-existence result for global (in time) solutions to this last Cauchy problem, then using the inverse transformation in \eqref{transformation u,v} we obtain a blow-up result for \eqref{CP semi scale inv} provided that \eqref{condition on delta} is satisfied.

Let us sketch the historical background of blow-up results for solutions to the Cauchy problem
\begin{align}\label{CP semi wave eq with general b and m^2}
\begin{cases}
w_{tt}-\Delta_x w +b(t) w_t+m^2(t) w=|w|^p, \qquad t>0, \; x\in \mathbb{R}^n, \\  w(0,x)=w_0(x), \quad x\in \mathbb{R}^n,\\  w_t(0,x)=w_1(x),\quad x\in \mathbb{R}^n,
\end{cases}
\end{align} that are related somehow to our scale-invariant model \eqref{CP semi scale inv}.

For the classical free wave equation with power nonlinearity on the right-hand side (which corresponds to the case $b(t)=m^2(t)\equiv 0$ in the notations of \eqref{CP semi wave eq with general b and m^2}), the critical exponent is the so-called Strauss exponent $p_0(n)$, which is defined as the positive root of the quadratic equation
\begin{align*}
(n-1)p^2 -(n+1)p-2=0.
\end{align*}

In particular, we refer to the classical works \cite{John79,Kato80,Glas81B,Sid84,Scha85,Jiao03,Yor06,Zhou07} and references therein for blow-up results when $1<p\leq p_0(n)$.

In \cite{TodYor01} and \cite{Zhang}  the massless case with constant coefficients in the linear part is considered. While in \cite{TodYor01} the authors have proved the blow-up of solutions in the case of sub-Fujita exponents (that is, for $1<p<p_{\Fuj}(n):=1+\frac{2}{n}$) by using a blow-up result for ordinary differential inequalities (cf. \cite[Proposition 3.1]{TodYor01}), in \cite{Zhang} it has been shown the same result (however, working in the more general frame of complete noncompact Riemannian manifold and including the critical case) introducing the nowadays called \emph{test function method}.

On the other hand, we have drastically less blow-up results concerning classical Klein-Gordon equations with power nonlinearity on the right-hand side. In \cite{TAO97}, for example, a blow-up result has been proved in space dimensions $n=1,2,3$ and for sub-Fujita exponents.

Let us now recall some results to semi-linear wave models (\ref{CP semi wave eq with general b and m^2}) with time-dependent dissipation $b(t)w_t$ and without any mass term, where $b(t)=\mu_1 (1+t)^{-\beta}$ with $\beta \in (-1,1]$ and $\mu_1>0$.

A blow-up result is proved in \cite{LinNashZhai12} by using the test function method, if $\beta\in (-1,1)$ and provided that $1<p\leq p_{\Fuj}(n)$.  Later in \cite{DabbLucRei13} the authors generalized this blow-up result to more general damping terms $b(t)w_t$ by using a \emph{modified test function method} (cf. \cite{DabbLucMod}). More precisely, the dissipation $b(t)w_t$, that is considered in \cite{DabbLucRei13}, is \emph{effective} according to the classification given in \cite{WiThe,WirthD}.

Afterwards the case $\beta=1$ was considered in \cite{Waka14A}. In this paper the author proves two blow-up results for the scale-invariant case, for $1<p\leq p_{\Fuj}(n)$ if $\mu_1>1$ and for $1<p\leq p_{\Fuj}(n+\mu_1-1)$ if $0<\mu_1\leq 1$, assuming a  suitable integral sign condition for the Cauchy data. Also in this case the test function method is used in order to prove these results. In particular, for $\mu_1>1$ the same result has been substantially already proved with the modified test function method in \cite{DabbLucMod}.

Then in \cite{DabbLucRei15} the special value $\mu_1=2$ is studied in the scale-invariant case.
More specifically, assuming nonnegative, non-trivial and compactly supported data, it is shown that the solution has to blow up in finite time for \[1<p\leq \max\{p_0(n+2),p_{\Fuj}(n)\}.\]
The proof of this result is based on Kato's lemma and it relies heavily on the fact that for this special value of the coefficient $\mu_1$ the structure of the model is somehow \textquotedblleft wave-like \textquotedblright. Indeed, through the so-called \emph{dissipative transformation}
\begin{align*}
\widetilde{w}(t,x)=(1+t)^{-\frac{\mu_1}{2}} w(t,x),
\end{align*} it is possible to transform this special scale-invariant model with power non-linearity in a free wave equation with non-linearity $(1+t)^{-(p-1)}|\widetilde{w}|^p$.

Recently, in \cite{LaiTakWak17} the authors took into consideration the scale-invariant wave equation with damping in the case in which, in some sense, we call the model hyperbolic-like. They have shown a nonexistence result for global (in time) solutions for \[ p_{\Fuj}(n)\leq p < p_0(n+2\mu_1)\,\,\,\mbox{ and}\,\,\, 0<\mu_1< \frac{n^2+n+2}{2(n+2)},\] where the upper bound for $\mu_1$ guarantees the non-emptiness of the range for $p$, by using an improved version of Kato's lemma, which allows to control the life-span of the solution from above (see \cite{Tak15}).

Finally, let us mention blow-up results which are known for the scale-invariant case when also the mass term is present. In \cite{NunPalRei16,Pal17} it is proved that the solution blows up for \[ 1<p\leq p_{\Fuj}\Big(n+\frac{\mu_1-1-\sqrt{\delta}}{2}\Big)\] assuming $\delta\geq 0$ and suitable sign conditions for the initial data (moreover, in \cite{Pal17}, also the compactness of the supports of data is required). Although the range of $p$, for which the solution is not globally in time defined, is the same in both results, a different approach is used in the corresponding proofs. While in \cite{NunPalRei16} the test function method is considered, in \cite{Pal17} it is employed a proper modification of the blow-up result for ordinary differential inequalities introduced first in \cite{TodYor01} for the constant coefficients case and adapted then in \cite{Nishi11} for coefficients $b(t)=\mu_1(1+t)^{-\beta}$, where $\beta\in [0,1)$.

Furthermore, in \cite{NunPalRei16} a further nonexistence result is shown in the case in which the coefficients of the damping and mass term satisfy $\delta=1$. In more detail,  it is proved that the solution blows up in finite time (using Kato's lemma) provided that \[ 1<p\leq \max\Big\{p_0(n+\mu_1), p_{\Fuj}\Big(n+\frac{\mu_1}{2}-1\Big)\Big\}\] and the data are nonnegative and compactly supported. In particular the case $\mu_1=2,\mu_2^2=0$ (already considered in \cite{DabbLucRei15}) is included as a special case there.

In the present paper our goal is to prove  blow-up results that generalize or partially improve the results from \cite{Waka14A,DabbLucRei15,NunPalRei16,LaiTakWak17,Pal17}. After the completion of this paper we received the preprint \cite{IkedaSob17}, where a blow-up result for the semi-linear wave equation with just scale-invariant damping is improved by using a different method. Indeed, as we will explain in the concluding remarks, in that paper the authors develop a technique in the special case $\mu_2=0$ which provides a result that covers also cases that we are not able to investigate with the approach of the present paper.

\subsection{Notations} \label{subsection notations}

In this paper, we write $f \lesssim g$, when there exists a constant $C \geq 0$ such that $f \leq Cg$. On the one hand we write $f \asymp g$ when $g \lesssim  f \lesssim g$. On the other hand we write $f\simeq g$ when $f=Cg$ for some positive constant $C$.

As in the introduction we denote throughout the article by $p_{\Fuj}(n)$ and  $p_0(n)$ the Fujita exponent and the Strauss exponent, respectively.

For sake of brevity, we put
\begin{align*}
\phi(\tau):=\tfrac{\tau^{\ell+1}}{\ell+1} \qquad \mbox{for}\,\, \tau \geq 0.
\end{align*}
Moreover, if $a(t):=(1+t)^\ell$ is the time-dependent speed of propagation for the transformed Cauchy problem \eqref{CP semi Tric eq}, then we denote by $A(t)$ the primitive of $a$ that vanishes for $t=0$, namely
\begin{align*}
A(t):=\int_0^t a(s) ds = \tfrac{1}{\ell+1}\big((1+t)^{\ell+1}-1\big)= \phi(1+t)-\phi(1).
\end{align*}
We will employ the notations $B_r$ and $B_r(x)$ for the open ball with radius $r>0$ centered at the origin  and at a point $x\in \mathbb{R}^n$, respectively.

Furthermore, we denote by $M(u)$ the Hardy-Littlewood maximal function for any $u\in L^1_{\loc}(\mathbb{R}^n)$ (see Section \ref{Section maximal function}).

Finally, by $L^{p,\infty}(X)$ it will be denoted the \emph{weak $L^p$ space} on the measure space $(X,\mathfrak{M},\mu)$ (see Section \ref{Section weak Lp}).

\subsection{Main results}\label{Section Main Theorems}

Let us state the main blow-up results that we are going to prove in the present article.

\begin{thm}\label{Main theo} Assume that $u \in \mathcal{C}^2\left( [0,T) \times \mathbb{R}^n \right)$ is a classical solution to
\eqref{CP semi Tric eq} with $\ell\geq 0, k>-2$ and nonnegative, compactly supported initial data $(u_0,u_1) \in \mathcal{C}^2(\mathbb{R}^n) \times \mathcal{C}^1(\mathbb{R}^n) $ such that $u_0$ is not identically $0$.

 If the exponent $p>1$ satisfies one of the following conditions:
\begin{align}
p&< p_{\NE}(n;\ell,k):= \max\big\{p_0(n;\ell,k),p_1(n;\ell,k)\big\},\label{condition on p subcritical}
\\ p &=p_{\NE}(n;\ell,k)=p_1(n;\ell,k), \label{condition on p critical p1}
\\  p &=p_{\NE}(n;\ell,k)=p_0(n;\ell,k), \qquad \mbox{if} \;\; n\geq 2, \label{condition on p critical p0}
\end{align} where \[ p_1(n;\ell,k):=\frac{(\ell+1)n+k+1}{(\ell+1)n-1}\] and $p_0(n;\ell,k)$ is the positive root of the quadratic equation
\begin{align}
((\ell+1)n-1)p^2-((\ell+1)n+2k+1-2\ell)p-2(\ell+1)=0, \label{equation for p_0(n,l,k)}
\end{align}
then $u$ blows up in finite time, that is, $T<\infty$.
\end{thm}

\begin{rmk} Using the same notations of the previous statement, we find in particular that for $\ell=k=0$ the exponent $p_0(n;\ell,k)$ coincides with the Strauss exponent $p_0(n)$.
Moreover, since $p_1(n;0,0)=\frac{n+1}{n-1}< p_0(n)$ for any $n\geq 2$, we find the well-known blow-up result for the free wave equation with power non-linearity in the special case $\ell=k=0$.
\end{rmk}
\begin{rmk} Let us  underline that the discriminant of the second order equation \eqref{equation for p_0(n,l,k)} is always positive.
Therefore, thanks to Descartes' rule it follows that equation \eqref{equation for p_0(n,l,k)} has one positive root and one negative root.
Consequently, the second order inequality
\begin{align}\label{inequality for p_0(n,l,k)}
((\ell+1)n-1)p^2-((\ell+1)n+2k+1-2\ell)p-2(\ell+1)<0,
\end{align} gives actually an upper bound for $p>1$.
Moreover, we can observe with the same type of argument that $p_0(n;\ell,k)>1$ since $k>-2$.
\end{rmk}
Using the transformation \eqref{transformation u,v} we may derive from Theorem \ref{Main theo} the following corollaries for the scale-invariant model with power non-linearity.

\begin{cor}\label{Blow up cor Tric eq} Let $n\geq 1$ and let $\mu_1$ and $\mu_2^2$ be nonnegative constants satisfying $\delta\in (0,1]$.
Let us assume that $p>1$ satisfies one of the following conditions:
\begin{align*}
p&<p_{\mu_1,\mu_2}(n):=\max\Big\{p_0(n+\mu_1),p_{\Fuj}\Big(n+\tfrac{\mu_1-1}{2}-\tfrac{\sqrt{\delta}}{2}\Big)\Big\}, \\
 p&=p_{\mu_1,\mu_2}(n)= p_{\Fuj}\Big(n+\tfrac{\mu_1-1}{2}-\tfrac{\sqrt{\delta}}{2}\Big) ,\\
  p&=p_{\mu_1,\mu_2}(n)=p_0(n+\mu_1),  \qquad \mbox{if} \,\, n=2.
\end{align*} Finally, let $v\in \mathcal{C}^2([0,T)\times\mathbb{R}^n)$ be a classical solution to the Cauchy problem \eqref{CP semi scale inv} with nontrivial and compactly supported initial data $(v_0,v_1) \in \mathcal{C}^2(\mathbb{R}^n) \times \mathcal{C}^1(\mathbb{R}^n) $ such that \begin{align*}
v_0\geq 0, \qquad v_1+\Big(\tfrac{\mu_1-1+\sqrt{\delta}}{2}\Big)v_0 \geq 0.
\end{align*}
Then $v$ blows up in finite time, that is, $T<\infty$.
\end{cor}

The condition on $k$ in Theorem \ref{Main theo} implies that the upper bound for the exponent $p$ in \eqref{CP semi Tric eq} is actually larger than 1. Nevertheless, when we consider \eqref{CP semi scale inv} and, consequently, $\ell$ and $k$ are defined through \eqref{definition ell and k}, the condition $k>-2$ makes no sense anymore, since $k$ depends on $p$. Indeed, the inequalities, which imply the conditions on $p$ in \eqref{CP semi scale inv}, are different from those, that imply the conditions for $p$ in  \eqref{CP semi Tric eq} for $k$ independent of $p$. More precisely, for $\ell$ and $k$ are defined by \eqref{definition ell and k} the conditions $p<p_0(n;\ell,k)$ and $p<p_1(n;\ell,k)$ can be written as $p<p_0(n+\mu_1)$ and $p<p_{\Fuj}\big(n+\tfrac{\mu_1-1}{2}-\tfrac{\sqrt{\delta}}{2}\big)$, respectively. For this reason the only necessary condition on $\mu_1$ and $\mu_2^2$ is \eqref{condition on delta}, which allows to choose a nonnegative $\ell$.

\begin{rmk} From Corollary \ref{Blow up cor Tric eq}  we see that the condition \eqref{condition on delta} on $\delta$ does not make our model \emph{hyperbolic-like}, from the point of view of the critical exponent, in all cases. Indeed, the upper bound for $p$, for which we can prove  a nonexistence result, feels the influence of two different terms. Depending on the dominant influence, we can classify the model \eqref{CP semi scale inv}
as \emph{parabolic-like} (the Fujita exponent is dominant) or \emph{hyperbolic-like} (the Strauss exponent is dominant), respectively.
\end{rmk}

Finally, from Corollary \ref{Blow up cor Tric eq} one may derive immediately the following two blow-up results for solutions to the wave equation with scale-invariant damping
and power non-linearity
 \begin{align}\label{CP semi scale inv mu2=0}
\begin{cases}
v_{\tau\tau}-\Delta v +\frac{\mu_1}{1+\tau} v_\tau=|v|^p, \qquad \tau>0, \; y\in \mathbb{R}^n, \\  v(0,y)=v_0(y), \quad y\in \mathbb{R}^n,\\  v_\tau(0,y)=v_1(y),\quad y\in \mathbb{R}^n.
\end{cases}
\end{align}
\begin{cor}\label{cor damping 1} Let $n\geq 1$ and $\mu_1\in [0,1)$.
Let us assume that $p>1$ satisfies one of the following conditions:
\begin{align*}
p&<p_{\mu_1}(n):=p_0(n+\mu_1),
\\
 p&=p_{\mu_1}(n)=p_0(n+\mu_1),  \qquad \mbox{if} \,\, n\geq 2, 
\end{align*} Moreover, let $v\in \mathcal{C}^2([0,T)\times\mathbb{R}^n)$ be a classical solution to the Cauchy problem \eqref{CP semi scale inv mu2=0} with nonnegative, nontrivial and compactly supported initial data $(v_0,v_1) \in \mathcal{C}^2(\mathbb{R}^n) \times \mathcal{C}^1(\mathbb{R}^n) $.
Then $v$ blows up in finite time, that is, $T<\infty$.
\end{cor}
\begin{rmk} In Corollary \ref{cor damping 1} we used the property $p_0(n)>p_{\Fuj}(n-1)$ for any $n>1$ to show that for $\mu_1\in [0,1)$ it holds
\begin{align*}
p_{\mu_1}(n)= \max\big\{p_0(n+\mu_1),p_{\Fuj}(n+\mu_1-1)\big\}=p_0(n+\mu_1).
\end{align*} Moreover, in the special case $n=1$ and $\mu_1=0$ we introduce as usual $p_0(1)=\infty$, since solutions blow up, in general, for any $p>1$.
\end{rmk}
\begin{cor}\label{cor damping 2}  Let $n\geq 1$ and $\mu_1\in (1,2]$.
Let us assume that $p>1$ satisfies one of the following conditions:
\begin{align*}
p&<p_{\mu_1}(n):= \max\big\{p_0(n+\mu_1),p_{\Fuj}(n)\big\}, \\
p&=p_{\mu_1}(n)= p_{\Fuj}(n) ,\\
 p&=p_{\mu_1}(n)=p_0(n+\mu_1) , \qquad \mbox{if} \,\, n=2.
\end{align*} Moreover, let $v\in \mathcal{C}^2([0,T)\times\mathbb{R}^n)$ be a classical solution to the Cauchy problem \eqref{CP semi scale inv mu2=0} with nontrivial and compactly supported initial data $(v_0,v_1) \in \mathcal{C}^2(\mathbb{R}^n) \times \mathcal{C}^1(\mathbb{R}^n) $ such that
\begin{align*}
v_0\geq 0, \qquad v_1+(\mu_1-1)v_0 \geq 0.
\end{align*}
Then $v$ blows up in finite time, that is, $T<\infty$.
\end{cor}
\begin{rmk} In Corollaries \ref{cor damping 1} and \ref{cor damping 2} the exponent $p_{\mu_1}(n)$ is nothing but the exponent $p_{\mu_1,\mu_2}(n)$ for $\mu_2=0$.
\end{rmk}

\section{Overview on our approach}\label{Section our approach}
\setcounter{equation}{0}
Throughout this article we will consider the time-dependent function
\begin{align}\label{definiition of G(t)}
\mathscr{G}(t)=\int_{\mathbb{R}^n}u(t,x) dx,
\end{align} where $u$ is a classical solution of \eqref{CP semi Tric eq}.

Since we require compactly supported data, by the property of finite speed of propagation it follows that also $u$ is compactly supported with respect to the spatial variables for any time up to its life-span. Therefore, if we prove that $\mathscr{G}$ blows up in finite time, then $u$ blows up in finite time as well.

A fundamental tool to estimate the blow-up dynamic of the function $\mathscr{G}(t)$ is the so-called Kato's Lemma.

\begin{lemma} [Kato's Lemma]\label{Kato's lemma}
Let $p>1,$ $q \in \mathbb{R}$ and $F \in \mathcal{C}^2([0,T))$ be a positive function satisfying
the nonlinear ordinary differential inequality
\begin{align} \label{condition on 2nd derivative of F Kato}
\frac{d^2}{dt^2}F(t) \geq k_1(t+R)^{-q}(F(t))^p
\end{align}
for any $t \in [T_1,T),$ for some $k_1,R >0$ and $T_1 \in [0,T).$
\begin{itemize}
\item[\rm{(i)}] If it holds the inequality
\begin{align} \label{condition on  of F Kato}
F(t) \geq k_0(t+R)^{a} \qquad \mbox{for any} \,\,\, t \in [T_0,T),
\end{align}
 for some $a\geq 1$ satisfying $ a > \frac{q-2}{p-1}$
and for some $k_0>0$ and $T_0 \in [0,T),$ then $T < \infty$.
\item[\rm{(ii)}] Let $q \geq p+1$
in \eqref{condition on 2nd derivative of F Kato} and suppose that the constant $k_0=k_0(k_1)>0$ is sufficiently large. Then, if
\eqref{condition on  of F Kato} holds with $a = \frac{q-2}{p-1}$ for some $T_0 \in [0,T),$ then $T < \infty$.
\end{itemize}
\end{lemma}

For the proof of the previous result one can see \cite{Sid84,Yag05} for the subcritical case and \cite{Yor06,DabbLucRei15} for the critical case.

The strategy in the proof of Theorem \ref{Main theo} is to apply Kato's lemma to the function $\mathscr{G}(t)$ in order to prove that the life span of such a function and, consequently, the life span of the solution $u$, has to be necessarily finite for $p$ as in the statement. We prove the sub-critical and the critical case of Theorem \ref{Main theo} in Sections \ref{Section subcritical case} and \ref{Section Blowup critical case}, respectively. In Section \ref{Section repres formula Tric eq} we derive an explicit integral representation formula for solutions to the inhomogeneous linear Cauchy problem in 1d related to \eqref{CP semi Tric eq}. Here we strongly follow Yagdjian's Integral Transform Approach of the papers \cite{Yag04,YagGal09,Yag15,Yag16}. Finally, in Section \ref{Section corollaries for the scale inv case} we prove Corollary \ref{Blow up cor Tric eq}. Some concluding remarks and open problems (see Section \ref{Sectionconcludingremarks}) complete the paper.

\section{Subcritical case}\label{Section subcritical case}
\setcounter{equation}{0}

Let us prove the blow-up result for \eqref{CP semi Tric eq} in the subcritical case, that is, when we employ Kato's lemma for the case in which the exponent $a$ in \eqref{condition on  of F Kato} satisfies $a>\frac{q-2}{p-1}$. This condition corresponds to the requirement  \eqref{condition on p subcritical} in the statement of Theorem \ref{Main theo}.

\begin{proof}[Proof of Theorem \ref{Main theo} in the subcritical case]
Let $u$ be the classical local (in time) solution to the Cauchy problem \eqref{CP semi Tric eq}.
We define  the function $\mathscr{G}=\mathscr{G}(t)$ as in \eqref{definiition of G(t)}.
Let us choose a positive constant $R$ such that $\supp u_0, \supp u_1 \subset B_R$.

In the following we will employ the functions $A=A(t)$ and $\phi=\phi(t)$ are introduced in Section \ref{subsection notations}.

Thanks to the property of finite speed of propagation, we have
\begin{align}\label{compact support property solution}
\supp u(t,\cdot)\subset B_{R+A(t)} \qquad \mbox{for any} \,\,  t>0.
\end{align}

Therefore,
\begin{align*}
\frac{d^2 \mathscr{G}}{dt^2}(t)&=\int_{\mathbb{R}^n}u_{tt}(t,x) dx =(1+t)^{2\ell} \int_{\mathbb{R}^n}\Delta u(t,x) dx+(\ell+1)^2(1+t)^k\int_{\mathbb{R}^n}|u(t,x)|^p dx\\ &=(\ell+1)^2(1+t)^k\int_{\mathbb{R}^n}|u(t,x)|^p dx,
\end{align*} where in the last equality we used the divergence theorem and \eqref{compact support property solution}.

In particular, from the previous equality we can derive the positivity of $\mathscr{G}$. Indeed, using the convexity of $\mathscr{G}$ we get
\begin{align*}
\mathscr{G}(t)
\geq \mathscr{G}(0)+t \, \frac{d \mathscr{G}}{dt}(0) = \int_{\mathbb{R}^n} u_0(x) dx +t  \int_{\mathbb{R}^n} u_1(x) dx >0 \qquad \mbox{for} \,\, t\geq 0\, .
\end{align*}

By using Jensen's inequality we find for large $t$
\begin{align}
\frac{d^2 \mathscr{G}}{dt^2}(t)=(\ell+1)^2(1+t)^k\int_{\mathbb{R}^n}|u(t,x)|^p dx&\gtrsim (1+t)^k (R+A(t))^{-n(p-1)}\Big|\int_{\mathbb{R}^n}u(t,x) dx\Big|^p \notag\\ & \simeq (R+t)^{k-(\ell+1)n(p-1)} (\mathscr{G}(t))^p. \label{condition on G'' Kato}
\end{align}

In order to apply Lemma \ref{Kato's lemma} we need to derive a lower bound for $\mathscr{G}$ for large times. To get this bound from below we shall introduce a second time-dependent function $\mathscr{G}_1=\mathscr{G}_1(t)$ defined as the integral over $\mathbb{R}^n$ of the product between $u$ and a function $\psi=\psi(t,x)$. Therefore, let us choose the $t$-dependent factor and the $x$-dependent factor of $\psi$.

For this reason we consider the function $\lambda=\lambda(t)$ with
\begin{align}
\lambda(t) &=C_\ell (1+t)^{\frac{1}{2}}\mathcal{K}_{\nu}(\phi(1+t)) \notag \\ &=C_\ell (1+t)^{\frac{1}{2}} \int_0^\infty \exp\Big(-\tfrac{(1+t)^{\ell+1}}{\ell+1}\cosh z\Big)\cosh(\nu z) dz \qquad \mbox{with} \;\; \nu=\tfrac{1}{2(\ell+1)},\label{lambda function definition}
\end{align}  where $\mathcal{K}_\nu$ denotes the modified Bessel function of second kind (cf. \cite[Section 9.6]{AS84}) and the multiplicative constant $C_\ell$ is chosen in such a way that $\lambda(0)=1$.

The function $\lambda=\lambda(t)$ is a slight modification of the corresponding function considered in \cite{HeWittYin17}. More specifically, the function $\lambda$ is a translation of the function which is considered there (of course with different multiplicative constants).

The function $\lambda$, defined by \eqref{lambda function definition}, satisfies the equation
\begin{align*}
\lambda'' (t)-(1+t)^{2\ell}\lambda(t)=0.
\end{align*}

 This is a consequence of the fact that $\mathcal{K}_\nu$ is a solution of the modified Bessel differential equation
$$\Big(t^2 \frac{d^2}{dt^2}+t\frac{d }{dt} -(t^2+\nu^2)\Big)\mathcal{K}_\nu (t)=0.$$

Furthermore, the function $\lambda$ satisfies the following properties:
\begin{itemize}
\item $\lambda$ and $-\lambda'$ are decreasing functions and $\displaystyle{\lim_{t\to\infty}\lambda(t)=\lim_{t\to\infty}\lambda'(t)=0}$;
\item $|\lambda'(t)|\asymp \lambda(t) (1+t)^\ell$ uniformly for all $t\geq 0$.
\end{itemize}
For the proof of these two properties one can see \cite{HongLi96} (even though there is only the case $\ell\in\mathbb{N}$ considered, the same proof is valid for any real number $\ell\geq 0$).

Let us choose the function $$\varphi=\varphi(x)=\int_{\mathbb{S}^{n-1}}e^{x\cdot \omega} d\sigma_\omega.$$
The function $\varphi$ has the following properties:
\begin{itemize}
\item $\varphi\in \mathcal{C}^\infty (\mathbb{R}^n)$ and $\partial^\alpha_x \varphi(x)=\int_{\mathbb{S}^{n-1}}\omega^\alpha e^{x\cdot \omega} d\sigma_\omega$ for any multi-index $\alpha$, and in particular $\Delta \varphi=\varphi$;
\item $\varphi>0$ on $\mathbb{R}^n$.
\end{itemize}

The first property follows by the compactness of the unit sphere $\mathbb{S}^{n-1}$, while Cauchy-Schwartz inequality and the monotonicity of the exponential function imply the second one.

 Finally, we introduce the function
\begin{align*}
\psi(t,x)=\lambda(t)\varphi(x).
\end{align*} Due to the properties of both functions $\lambda$ and $\varphi$ we obtain that $\psi$ is a solution of the homogeneous linear equation
\begin{align*}
\psi_{tt}-(1+t)^{2\ell}\Delta\psi =0.
\end{align*}

As we anticipated, we define $$\mathscr{G}_1(t):=\int_{\mathbb{R}^n}u(t,x) \psi(t,x) dx.$$

So, using H\"{o}lder's inequality to estimate $\frac{d^2 \mathscr{G}}{dt^2}$ we get
\begin{align}
\frac{d^2 \mathscr{G}}{dt^2}(t)&=(\ell+1)^2(1+t)^k\int_{\mathbb{R}^n}|u(t,x)|^p dx \gtrsim (1+t)^k |\mathscr{G}_1(t)|^p\Big(\int_{B_{R+A(t)}(0)}\psi(t,x)^\frac{p}{p-1} dx\Big)^{-(p-1)}\label{estimate from below of G''}.
\end{align}

We estimate separately the two factors that appear in the right-hand side of \eqref{estimate from below of G''}.
Let us start with the integral
\begin{align*}
\int_{B_{R+A(t)}(0)}\psi(t,x)^\frac{p}{p-1} dx=\lambda(t)^{\frac{p}{p-1}}\int_{B_{R+A(t)}(0)}\varphi(x)^\frac{p}{p-1} dx.
\end{align*}

According to  \cite[page 364]{Yor06} the following asymptotic estimate holds:
\begin{align*}
\varphi(x) \simeq |x|^{-\frac{n-1}{2}}e^{|x|} \qquad \mbox{as} \,\, |x|\to \infty.
\end{align*}
The continuity of the function $\varphi$ implies the estimate
\begin{align*}
|\varphi(x)|\lesssim (1+|x|)^{-\frac{n-1}{2}}e^{|x|}
\end{align*}  uniformly for any $x\in\mathbb{R}^n$.


Therefore,
\begin{align*}
\int_{B_{R+A(t)}(0)}& \varphi(x)^\frac{p}{p-1} dx  \lesssim \int_0^{R+A(t)}(1+r)^{-\frac{n-1}{2}\cdot \frac{p}{p-1}} e^{\frac{p r}{p-1}}r^{n-1} dr \\
 \lesssim & \,\, e^{\frac{p}{p-1}\cdot\frac{R+A(t)}{2}}\!\int_0^{(R+A(t))/2}\!(1+r)^{n-1-\frac{n-1}{2}\cdot \frac{p}{p-1}}  dr +(R+A(t))^{n-1-\frac{n-1}{2}\cdot \frac{p}{p-1}}e^{\frac{p}{p-1}\cdot(R+A(t))}.
\end{align*}
For the integral over $[0,(R+A(t))/2]$ it is possible to find the same estimate we got for the integral over $[(R+A(t))/2,R+A(t)]$.

Indeed,
\begin{align*}
\int_0^{(R+A(t))/2}(1+r)^{n-1-\frac{n-1}{2}\cdot \frac{p}{p-1}}  dr \lesssim \begin{cases} (R+A(t))^{n-\frac{n-1}{2}\cdot \frac{p}{p-1}} & \mbox{if} \,\,\, n-\frac{n-1}{2}\cdot \frac{p}{p-1}>0 \\ \log (R+A(t)) & \mbox{if} \,\,\, n-\frac{n-1}{2}\cdot \frac{p}{p-1}=0 \\ 1 & \mbox{if} \,\,\, n-\frac{n-1}{2}\cdot \frac{p}{p-1}<0 \end{cases}
\end{align*} and, consequently,
\begin{align*}
e^{\frac{p}{p-1}\cdot\frac{R+A(t)}{2}}\!\int_0^{(R+A(t))/2}\!(1+r)^{n-1-\frac{n-1}{2}\cdot \frac{p}{p-1}}  dr  &=e^{\frac{p}{p-1}\cdot(R+A(t))} e^{-\frac{p}{p-1}\cdot\frac{R+A(t)}{2}} \!\int_0^{(R+A(t))/2}\!(1+r)^{n-1-\frac{n-1}{2}\cdot \frac{p}{p-1}}  dr \\ &\lesssim e^{\frac{p}{p-1}\cdot(R+A(t))} 
 (R+A(t))^{n-1-\frac{n-1}{2}\cdot \frac{p}{p-1}}.
\end{align*}

Summarizing, we proved
\begin{align*}
\Big(\int_{B_{R+A(t)}(0)}\psi(t,x)^\frac{p}{p-1} dx \Big)^{p-1} &\lesssim \lambda(t)^p \,e^{\,p(R+A(t))}  \,(R+A(t))^{(n-1)(p-1)-\frac{n-1}{2} p} .
\end{align*}

Then, using the asymptotic behavior of modified Bessel function of second kind  $$\mathcal{K}_\nu (t) = \sqrt{\frac{\pi}{2t}} e^{-t}\big(1+O(t^{-1})\big) \qquad \mbox{as} \;\; t\to \infty,$$ (see for example \cite[formula 9.7.2]{AS84}), then for large $t$ we may estimate
\begin{align*}
\lambda(t)\asymp  (1+t)^{\frac{1}{2}} \phi(1+t)^{-\frac{1}{2}} e^{-\phi(1+t)}\simeq (1+t)^{-\frac{\ell}{2}} e^{-\phi(1+t)}.
\end{align*}

Thus,
\begin{align}
\Big(\int_{B_{R+A(t)}(0)}\psi(t,x)^\frac{p}{p-1} dx \Big)^{p-1} &\lesssim (1+t)^{-\frac{\ell p}{2}} e^{-p\phi(1+t)} e^{p(R+A(t))}  (R+A(t))^{(n-1)(p-1)-\frac{n-1}{2}\, p}\notag \\ & \simeq (1+t)^{-\frac{\ell p}{2}} (R+A(t))^{(n-1)(p-1)-\frac{n-1}{2}\, p}. \label{upper bound for psi integral}
\end{align}

Let us derive now a lower bound for $\mathscr{G}_1(t)$. We point out that this estimate is slightly easier in our case than that one in \cite[Lemma 2.3]{HeWittYin17} because of the non zero behavior of the speed of propagation $a(t)=(1+t)^\ell$ in a neighborhood of $0$.


Integrating by parts we get
\begin{align*}
0 &\leq \int_0^t \int_{\mathbb{R}^n}(\ell+1)^2 (1+s)^k |u(s,x)|^p \psi(s,x) dx ds = \int_0^t \int_{\mathbb{R}^n}\big(u_{ss}(s,x)-(1+s)^{2\ell}\Delta u(s,x)\big) \psi(s,x) dx ds  \\ &=\int_0^t \int_{\mathbb{R}^n}u(s,x)\underbrace{\big(\psi_{ss}(s,x)-(1+s)^{2\ell}\Delta \psi(s,x)\big)}_{=0}  dx ds+ \int_{\mathbb{R}^n} \big(u_s(s,x)\psi(s,x)-u(s,x)\psi_s(s,x)\big) dx \,\, \Big|^{s=t}_{s=0}\\&= \int_{\mathbb{R}^n} \big(u_t(t,x)\psi(t,x)-u(t,x)\psi_t(t,x)\big) dx -\int_{\mathbb{R}^n} \big(u_1(x)-\lambda'(0)u_0(x)
\big)\varphi(x) dx,
\end{align*} where in the second equality for the integration by parts with respect to $x$ we may neglect the boundary integral because of \eqref{compact support property solution}.

Since we have assumed nonnegative data (the first data is not identically zero) and $\lambda'(0)<0$ it follows
\begin{align*}
\int_{\mathbb{R}^n} \big(u_t(t,x)\psi(t,x)-u(t,x)\psi_t(t,x)\big) dx \geq \int_{\mathbb{R}^n} \big(u_1(x)-\lambda'(0)u_0(x)
\big)\varphi(x) dx =C >0.
\end{align*}

On the other hand
\begin{align*}
\int_{\mathbb{R}^n} \big(u_t(t,x)\psi(t,x)-u(t,x)\psi_t(t,x)\big) dx &=\int_{\mathbb{R}^n} \Big(\tfrac{\partial}{\partial t}\big(u(t,x)\psi(t,x)\big)-2u(t,x)\psi_t(t,x)\Big) dx 
\\ &= \frac{d\mathscr{G}_1}{dt}(t)+2|\lambda'(t)|\int_{\mathbb{R}^n} u(t,x)\varphi(x)dx.
\end{align*} We have already mentioned the inequality $|\lambda'(t)|\leq C_1 \lambda(t)(1+t)^\ell$ for $t\geq 0$ which holds for some positive constant $C_1$. Therefore,
\begin{align*}
C &\leq   \frac{d\mathscr{G}_1}{dt}(t)+2 C_1 \lambda(t)(1+t)^{\ell}\int_{\mathbb{R}^n} u(t,x)\varphi(x)dx= \frac{d\mathscr{G}_1}{dt}(t)+2 C_1 (1+t)^{\ell}\mathscr{G}_1(t).
\end{align*}

Multiplying both sides of the previous inequality by $e^{2 C_1 A(t)}$ we get
\begin{align*}
\frac{d}{dt}\big(e^{2C_1 A(t)}\mathscr{G}_1(t)\big)\geq C e^{2C_1 A(t)}.
\end{align*}  After integration over $[0,t]$ it follows
\begin{align*}
e^{2C_1 A(t)}\mathscr{G}_1(t)-\mathscr{G}_1(0) &\geq C \int_0^t e^{2C_1 A(s)} ds \geq \frac{C}{2 C_1}(1+t)^{-\ell}\int_0^t 2C_1 (1+s)^\ell e^{2C_1 A(s)} ds \\ & = \frac{C}{2 C_1}(1+t)^{-\ell} (e^{2C_1 A(t)}-1).
\end{align*}
Hence,
\begin{align*}
\mathscr{G}_1(t)&\geq e^{-2C_1 A(t)}\mathscr{G}_1(0)+\frac{C}{2 C_1}(1+t)^{-\ell} -\frac{C}{2 C_1}(1+t)^{-\ell}e^{-2C_1 A(t)}\geq \frac{C}{2 C_1}(1+t)^{-\ell} (1-e^{-2C_1 A(t)}).
\end{align*}
Summarizing, for large $t$ we have
\begin{align}\label{lower bound G_1}
\mathscr{G}_1(t) \gtrsim (1+t)^{-\ell}.
\end{align}
Combining \eqref{estimate from below of G''}, \eqref{upper bound for psi integral} and \eqref{lower bound G_1}, then for large $t$ we arrive at
\begin{align}
\frac{d^2 \mathscr{G}}{dt^2}(t)&\gtrsim (1+t)^{k-\frac{\ell p}{2}}(R+A(t))^{-(n-1)(p-1)+\frac{n-1}{2} \,p}\gtrsim (R+t)^{k-\frac{\ell p}{2}-(n-1)(\frac{p}{2}-1)(\ell+1)} .  \label{condition on G'' Kato, pol growth}
\end{align}

Integrating twice the previous relation, we get
\begin{align}
\mathscr{G}(t)&\gtrsim \mathscr{G}(0)+\frac{d \mathscr{G}}{dt}(0)\, t+(R+t)^{\max\{k-\frac{\ell p}{2}-(n-1)(\frac{p}{2}-1)(\ell+1)+2,1\}} \notag\\ &\gtrsim (R+t)^{\max\{k-\frac{\ell p}{2}-(n-1)(\frac{p}{2}-1)(\ell+1)+2,1\}} \qquad \mbox{for large}\,\, t. \label{condition on G Kato}
\end{align}

Finally, we can apply Kato's lemma. So from \eqref{condition on G'' Kato} and \eqref{condition on G Kato} we obtain that $\mathscr{G}(t)$ (and, consequently, $u$) blows up provided that at least one of the following condition is fulfilled:
\begin{align}
k-\frac{\ell p}{2}-(n-1)\Big(\frac{p}{2}-1\Big)(\ell+1)+2>-\frac{k+2}{p-1}+(\ell+1)n,  \label{1st condition p blow up Tric}\\
 1>-\frac{k+2}{p-1}+(\ell+1)n. \label{2nd condition p blow up Tric}
\end{align}

Since the condition \eqref{1st condition p blow up Tric} is equivalent to the quadratic inequality \eqref{inequality for p_0(n,l,k)} and \eqref{2nd condition p blow up Tric} corresponds to $p<p_1(n;\ell,k)$. The proof is completed.
\end{proof}
\begin{rmk}
Let us underline explicitly that in the previous proof the case $\ell=0$ can also be included.
In this case it is enough to consider  $\lambda(t)=e^{-t}$ instead of the function defined by \eqref{lambda function definition}. Indeed $e^{-t}$ satisfies all properties that we used in the proof for $\ell=0$. Moreover, this special case coincides exactly with the proof of Theorem 1.1 in \cite{Yor06}.
\end{rmk}
\begin{rmk}\label{remark log term} In the previous proof we applied Lemma \ref{Kato's lemma} \rm{(i)}, and, consequently, we do not need to take care of the multiplicative constant in \eqref{condition on G Kato}. However, when we are in the critical case, that is, we shall apply  Lemma \ref{Kato's lemma} \rm{(ii)}, the situation is quite different. Thus, rather than to estimate quantitatively the unexpressed multiplicative constant in \eqref{condition on G Kato}, following the approach of \cite{Yor06} we will obtain an improved version of \eqref{condition on G Kato} with a further logarithmic factor. In such a way we can immediately apply the second part of Lemma \ref{Kato's lemma}, since we can make the logarithmic term arbitrarily large considering sufficiently large times. However, before proving this improvement of \eqref{condition on G Kato}, we derive firstly an integral representation formula for a one spatial dimensional problem related to \eqref{CP semi Tric eq}, that will be helpful in the proof of the case \eqref{condition on p critical p0}.
\end{rmk}

\section{An integral representation formula for solutions to an inhomogeneous linear wave equation in 1d}  \label{Section repres formula Tric eq}
\setcounter{equation}{0}

The goal of this section is to provide an explicit representation formula for solutions to the Cauchy problem
\begin{align}\label{Tric type CP 1d}
\begin{cases} u_{tt}-(1+t)^{2\ell}u_{xx}=f(t,x), & t>0,\, x\in \mathbb{R}, \\ u(0,x)= u_0(x), & x\in \mathbb{R}, \\ u_t(0,x)= u_1(x), & x\in \mathbb{R},
\end{cases}
\end{align}
via integral transformations and to investigate the properties of the corresponding kernel functions.

In the further considerations we will follow the approach of \cite{Yag04} and \cite{YagGal09}, where an explicit representation formula is derived through integral transformations for solutions to a Cauchy problem for a generalized Tricomi equation (containing the Gellerstedt operator) and for solutions to a Cauchy problem for the Klein-Gordon equation in the de Sitter spacetime, respectively.

\subsection{Inhomogeneous case with zero data} \label{section repres formula inhom probl}
The first step is to construct a family of fundamental solutions for the operator $$\mathcal{T}=\frac{\partial^2}{\partial t^2}-(1+t)^{2\ell}\frac{\partial^2}{\partial x^2}$$ related to the point $(t_0,x_0)$, where $t_0\geq 0$ and $x_0\in\mathbb{R}$, and supported in the forward cone
\begin{align*}
D_+(t_0,x_0)=\Big\{(t,x)\in \mathbb{R}^2 : t > t_0,\, |x-x_0|< \tfrac{1}{\ell+1}\big((1+t)^{\ell+1}-(1+t_0)^{\ell+1}\big)  \Big\}.
\end{align*}

In other words, we are looking for a distributional solution $E_+(t,x;t_0,x_0)$ of the equation
\begin{align*}
\mathcal{T} E_+= \partial_t^2 E_{+}-(1+t)^{2\ell}\partial_x^2 E_{+}= \delta_{t_0}(t)\otimes\delta_{x_0}(x)
\end{align*} with support contained in $D_+(t_0,x_0)$.

For this purpose, introducing the so-called \emph{characteristic coordinates}
\begin{equation*}
\begin{split}
r &= x+ \phi(1+t),\\
s &= x- \phi(1+t),
\end{split} 
\end{equation*} and, following the approach of \cite{Yag04,YagGal09}, we get the representation
\begin{align*}
E_+=E_+(t,x;t_0,x_0)=\begin{cases}c_\ell E(t,x;t_0,x_0) & \mbox{in}\,\, D_+(t_0,x_0),\\ 0 & \mbox{elsewhere},\end{cases}
\end{align*} with  $c_\ell=2^{-\frac{1}{\ell+1}}(\ell+1)^{-\frac{\ell}{\ell+1}}$ and
\begin{align}
E(t,x;t_0,x_0)=\big( (\phi(1+t)+\phi(1+t_0))^2-(x-x_0)^2\big)^{-\gamma}F\Big(\gamma,\gamma;1;\frac{(\phi(1+t)-\phi(1+t_0))^2-(x-x_0)^2}{(\phi(1+t)+\phi(1+t_0))^2-(x-x_0)^2}\Big), \label{definition E(t,x,t0,x0)}
\end{align} where $\gamma=\frac{\ell}{2(\ell+1)}$ and $F(\gamma,\gamma; 1;\cdot)$ denotes the Gauss hypergeometric function (cf. Section \ref{Section Gauss hypergeo funct}).

Let $f=f(t,x)$ $\in\mathcal{C}^2(\mathbb{R}^2$) be a source term in \eqref{Tric type CP 1d} having the properties
\begin{itemize}
\item $\supp f \subset \{(t,x)\in \mathbb{R}^2: t\geq 0\}$,
\item $\supp f(t,\cdot)$ is compact for any $t\geq 0$.
\end{itemize}

Then, following the approach of \cite{Yag04,YagGal09} we may write a solution of the Cauchy problem \eqref{Tric type CP 1d} with vanishing data $u_0=u_1=0$ in the following form:
\begin{align*}
u(t,x)&=\iint_{\mathbb{R}^2} E_+(t,x;b,y)\, f(b,y) \,d(b,y) \\
&= c_\ell \int_0^t \int_{x-(\phi(1+t)-\phi(1+b))}^{x+(\phi(1+t)-\phi(1+b))}  E(t,x-y;b,0)\, f(b,y) \, dy \,db \\
&= c_\ell \int_0^t  \int_{x-(\phi(1+t)-\phi(1+b))}^{x+(\phi(1+t)-\phi(1+b))}    f(b,y) \, \Big((\phi(1+t)+\phi(1+b))^2-(x-y)^2\Big)^{-\gamma} \\ & \qquad \qquad \qquad\times F\Big(\gamma,\gamma;1; \frac{(\phi(1+t)-\phi(1+b))^2-(x-y)^2}{(\phi(1+t)+\phi(1+b))^2-(x-y)^2}\Big) \, dy \,db.
\end{align*}

Using the property $E(t,-y;b,0)=E(t,y;b,0)$ we may also write
\begin{align*}
u(t,x)
&= c_\ell \int_0^t  \int_{0}^{\phi(1+t)-\phi(1+b)}   [ f(b,x-y)+f(b,x+y) ] \, E(t,y;b,0) \, dy \,db\\
&= c_\ell \int_0^t \int_{0}^{\phi(1+t)-\phi(1+b)} [ f(b,x-y)+f(b,x+y) ] \,\Big((\phi(1+t)+\phi(1+b))^2-y^2\Big)^{-\gamma}  \\ & \qquad \qquad \qquad\times F\Big(\gamma,\gamma;1; \frac{(\phi(1+t)-\phi(1+b))^2-y^2}{(\phi(1+t)+\phi(1+b))^2-y^2}\Big)\, dy \,db.
\end{align*}

It is important to underline that the kernel function of the previous  integral is nonnegative on the  domain of integration. Indeed, the argument of the Gauss hypergeometric function is an element of the interval $[0,1)$ for any $(b,y)$ in the domain of integration and, consequently, the function $F(\gamma,\gamma;1,\cdot)$ is there always positive, since the parameters $(\gamma,\gamma;1)$ are all positive. Actually, for the limit case $\ell=0$ we have $\gamma=0$, but also in this case the Gauss function, which is just the constant function $1$, is nonnegative.

The previously defined function $u$ is a solution of \eqref{Tric type CP 1d} in the classical sense (even in the weak sense if all integrals are defined).

\subsection{Homogeneous case}

Up to now we derived a representation formula for the solution of the linear inhomogeneous Cauchy problem \eqref{Tric type CP 1d}, in the case in which the initial data $u_0,u_1$ are identically zero. The next step is to derive a representation formula for solutions to the linear homogeneous Cauchy problem
\begin{align}\label{Tric type CP 1d lin hom}
\begin{cases}
u_{tt}-(1+t)^{2\ell} u_{xx}=0,& t>0,\,\, x\in \mathbb{R},\\
u(0,x)=u_0(x), & x\in \mathbb{R},\\
u_t(0,x)=u_1(x), & x\in \mathbb{R}.
\end{cases}
\end{align}
We shall study separately the cases $u_0=0$ and $u_1=0$.

The strategy to derive these representations is to reduce the corresponding homogeneous linear Cauchy problem to a suitable inhomogeneous Cauchy problem. Then, applying the result from Section \ref{section repres formula inhom probl} and rewriting the obtained expressions in a appropriate way, we will get the desired representation formulas.

But first we are going to prove the following preliminary lemma which is helpful in next sections. In the further consideration we use the abbreviation $A(t):=\phi(1+t)-\phi(1)$.
\begin{lemma} Let $E=E(t,x;t_0,x_0)$ be the function defined by \eqref{definition E(t,x,t0,x0)}. Then, the following formulas are valid:
\begin{align}
E(t,x;b,y)&=E(b,x;t,y) ,\label{symmetry E t,b} \\
E(t,x;b,y)&=E(t,y;b,x) ,\label{symmetry E x,y} \\
E(t,x;b,y)&=E(b,x-y;t,0) ,\label{space invariance E} \\
E(t,-x;b,0)&=E(b,x;t,0) ,\label{even E in x, y=0} \\
E(t,A(t)-A(b);b,0)&=2^{-2\gamma} (\ell+1)^{2\gamma}(1+t)^{-\frac{\ell}{2}}(1+b)^{-\frac{\ell}{2}},\label{E in y=A(t)-A(b)} \\
\frac{\partial E}{\partial y}(t,y;b,0)\Big|_{y=A(t)-A(b)} &= 2^{-2\gamma-3} \ell(\ell+2) (\ell+1)^{2\gamma}\,(1+t)^{-\frac{3}{2}\ell-1}(1+b)^{-\frac{3}{2}\ell-1}\big(\phi(1+t)-\phi(1+b)\big) ,\label{dE/dy in y=A(t)-A(b)} \\
E(A^{-1}(A(t)-y),y;t,0) &=2^{-2\gamma} (\ell+1)^{\gamma} (1+t)^{-\frac{\ell}{2}}(\phi(1+t)-y)^{-\gamma}, \label{E in y=A^-1(A(t)-y)} \\
\frac{\partial E}{\partial b}(t,y;b,0)\Big|_{b=A^{-1}(A(t)-y)}&= -2^{-2\gamma-1}(\ell+1)^{2\gamma}(\phi(1+t)-y)^{\gamma-1}  \phi(1+t)^{-\gamma-1} \big(\gamma(2\phi(1+t)-y)+\gamma^2 y\big).\label{dE/db for b=A^-1(A(t)-y)}
\end{align}

\end{lemma}

\begin{proof}
Properties \eqref{symmetry E t,b} to \eqref{even E in x, y=0} follow directly from the definition of $E(t,x;b,y)$.
Since $F(\gamma,\gamma;1;0)=1$ we get for $x=A(t)-A(b)=\phi(1+t)-\phi(1+b)$ the relation
\begin{align*}
E(t,x;b,0)= \big((\phi(1+t)+\phi(1+b))^2-(\phi(1+t)-\phi(1+b))^2\big)^{-\gamma}= 2^{-2\gamma}\big(\phi(1+t)\phi(1+b)\big)^{-\gamma}
\end{align*} and, thus, \eqref{E in y=A(t)-A(b)}.
The partial derivative of $E(t,y;b,0)$ with respect to $y$ is
\begin{align*}
\frac{\partial E}{\partial y}(t,y;b,0)
& =2\gamma y \big((\phi(1+t)+\phi(1+b))^2-y^2\big)^{-\gamma-1}F\Big(\gamma,\gamma;1;\frac{(\phi(1+t)-\phi(1+b))^2-y^2}{(\phi(1+t)+\phi(1+b))^2-y^2}\Big) \\ & \quad + \big((\phi(1+t)+\phi(1+b))^2-y^2\big)^{-\gamma}\gamma^2 F\Big(\gamma+1,\gamma+1;2;\frac{(\phi(1+t)-\phi(1+b))^2-y^2}{(\phi(1+t)+\phi(1+b))^2-y^2}\Big)\\ & \qquad \times\frac{\partial}{\partial y}  \Big(\frac{(\phi(1+t)-\phi(1+b))^2-y^2}{(\phi(1+t)+\phi(1+b))^2-y^2} \Big),
\end{align*} where we used the rule \eqref{derivative hypergeometric function} for the derivative of $F(\gamma,\gamma;1;\cdot)$. The partial derivative with respect to $y$ in the last line of the previous formula is
\begin{align*}
\frac{\partial}{\partial y}  \bigg(\frac{(\phi(1+t)-\phi(1+b))^2-y^2}{(\phi(1+t)+\phi(1+b))^2-y^2} \bigg)
 &=-\frac{8\phi(1+t)\phi(1+b)y}{\big(\big(\phi(1+t)+\phi(1+b)\big)^2-y^2\big)^2}.
\end{align*} This implies
\begin{align*}
\frac{\partial}{\partial y}  \bigg(\frac{(\phi(1+t)-\phi(1+b))^2-y^2}{(\phi(1+t)+\phi(1+b))^2-y^2} \bigg)\Big|_{y=A(t)-A(b)}&=-\frac{8\phi(1+t)\phi(1+b)\big(\phi(1+t)-\phi(1+b)\big)}{\big(4\phi(1+t)\phi(1+b)\big)^2}\\&=-\frac{\phi(1+t)-\phi(1+b)}{2\phi(1+t)\phi(1+b)}.
\end{align*}

Consequently, combining the previous relations we derive \eqref{dE/dy in y=A(t)-A(b)} in the following way:
\begin{align*}
\frac{\partial E}{\partial y}(t,y;b,0)&\Big|_{y=A(t)-A(b)}\notag\\ &= 2\gamma \big(\phi(1+t)-\phi(1+b)\big) \big(4 \phi(1+t)\phi(1+b)\big)^{-\gamma-1}F(\gamma,\gamma;1;0) \notag \\ & \quad -\tfrac{\gamma^2}{2}\big(4 \phi(1+t)\phi(1+b)\big)^{-\gamma} F(\gamma+1,\gamma+1;2;0)\,\frac{\phi(1+t)-\phi(1+b)}{\phi(1+t)\phi(1+b)} \notag \\ & = 2^{-2\gamma-1} \gamma (1-\gamma)\,\big(\phi(1+t)\phi(1+b)\big)^{-\gamma-1}\big(\phi(1+t)-\phi(1+b)\big) \notag \\ & = 2^{-2\gamma-3} \ell(\ell+2) (\ell+1)^{2\gamma}\,(1+t)^{-\frac{3}{2}\ell-1}(1+b)^{-\frac{3}{2}\ell-1}\big(\phi(1+t)-\phi(1+b)\big) .
\end{align*}

Let $A^{-1}(z)=((\ell+1)z+1)^{\frac{1}{\ell+1}}-1$ be the inverse function of $A$.

 Let us evaluate $E(b,y;t,0) $ and $\frac{\partial E}{\partial b}(b,y;t,0) $
for $b=A^{-1}(A(t)-y)=
\big((1+t)^{\ell+1}-(\ell+1)y\big)^{\frac{1}{\ell+1}}-1.$
Since for such value of $b$ it holds
\begin{align}
\phi(1+b)=\phi(1+t)-y, \label{phi(1+b)=phi(1+t)-y}
\end{align} we may conclude
\begin{align*}
E(A^{-1}(A(t)-y),y;t,0)&=\big((2\phi(1+t)-y)^2-y^2\big)^{-\gamma}F(\gamma,\gamma;1;0)\\& =2^{-2\gamma} (\ell+1)^{\gamma} (1+t)^{-\frac{\ell}{2}}(\phi(1+t)-y)^{-\gamma},
\end{align*} that is \eqref{E in y=A^-1(A(t)-y)}.

On the other hand, the partial derivative of $E(b,y;t,0)$ with respect to $b$ is
\begin{align}
\frac{\partial E}{\partial b}(b,y;t,0)&= -2\gamma(1+b)^\ell \big(\phi(1+b)+\phi(1+t)\big)\big((\phi(1+b)+\phi(1+t))^2-y^2\big)^{-\gamma-1}\notag\\ &\qquad\qquad \times F\Big(\gamma,\gamma;1;\frac{(\phi(1+b)-\phi(1+t))^2-y^2}{(\phi(1+b)+\phi(1+t))^2-y^2}\Big)\notag\\ &\quad +4\gamma^2\big((\phi(1+b)+\phi(1+t))^2-y^2\big)^{-\gamma}F\Big(\gamma+1,\gamma+1;2;\frac{(\phi(1+b)-\phi(1+t))^2-y^2}{(\phi(1+b)+\phi(1+t))^2-y^2}\Big)\notag\\ &\qquad\qquad \times    \frac{(1+b)^\ell \phi(1+t)\big(\phi(1+b)^2-\phi(1+t)^2+y^2\big)}{\big((\phi(1+b)+\phi(1+t))^2-y^2\big)^2},\label{dE/db}
\end{align} where in the second term we used again the rule \eqref{derivative hypergeometric function} and
\begin{align*}
\frac{\partial}{\partial b}\bigg(\frac{(\phi(1+b)-\phi(1+t))^2-y^2}{(\phi(1+b)+\phi(1+t))^2-y^2}\bigg)&=  \frac{4 (1+b)^\ell \phi(1+t)\big(\phi(1+b)^2-\phi(1+t)^2+y^2\big)}{\big((\phi(1+b)+\phi(1+t))^2-y^2\big)^2}.
\end{align*}

In particular, since for $b=A^{-1}(A(t)-y)$ the relations $$(1+b)^\ell=\big((1+t)^{\ell+1}-(\ell+1)y\big)^{2\gamma}=(\ell+1)^{2\gamma}(\phi(1+t)-y)^{2\gamma}$$ and \eqref{phi(1+b)=phi(1+t)-y} are satisfied, we obtain the following relation for this value of $b$ in \eqref{dE/db}:
\begin{align*}
\frac{\partial E}{\partial b}&(b,y;t,0)\Big|_{b=A^{-1}(A(t)-y)}\notag
\\&= -2\gamma (\ell+1)^{2\gamma}(\phi(1+t)-y)^{2\gamma} (2\phi(1+t)-y) \big((2\phi(1+t)-y)^2-y^2\big)^{-\gamma-1}\notag
\\ &\quad + 4\gamma^2 \big((2\phi(1+t)-y)^2-y^2\big)^{-\gamma-2}   (\ell+1)^{2\gamma}(\phi(1+t)-y)^{2\gamma} \phi(1+t)\notag
\\ &\qquad \times\big((\phi(1+t)-y)^2 -\phi(1+t)^2 +y^2\big)\notag
\\&= -2^{-2\gamma-1}(\ell+1)^{2\gamma}(\phi(1+t)-y)^{\gamma-1}  \phi(1+t)^{-\gamma-1}\big(\gamma(2\phi(1+t)-y)+\gamma^2 y\big).
\end{align*} This is exactly \eqref{dE/db for b=A^-1(A(t)-y)} and the proof is completed.
\end{proof}
\subsubsection{Homogeneous case with vanishing first data}
\begin{thm}\label{thm repres formula 1d lin tric u_0=0}
The solution $u=u(t,x)$ of the Cauchy problem
\begin{align}\label{linear CP with u0=0}
\begin{cases}
u_{tt}-(1+t)^{2\ell} u_{xx}=0,& t>0,\,\, x\in \mathbb{R},\\
u(0,x)=0, & x\in \mathbb{R},\\
u_t(0,x)=u_1(x), & x\in \mathbb{R},
\end{cases}
\end{align}
with $u_1\in C^\infty_0 (\mathbb{R})$ can be represented as follows:
\begin{align*}
u(t,x)= \int_{0}^{A(t)} \big(u_1(x-y)+u_1(x+y)\big)K_1(t,y) \, dy \, ,
\end{align*} where the kernel $K_1=K_1(t,y)$ is defined by
\begin{align}
K_1(t,y)&=c_\ell E(t,y;0,0) \label{def K1(t,y) via E}\\&= c_\ell \big((\phi(1+t)+\phi(1))^2-y^2\big)^{-\gamma}F\Big(\gamma,\gamma;1; \frac{A(t)^2-y^2}{(\phi(1+t)+\phi(1))^2-y^2}\Big).\label{def K1(t,y)}
\end{align}
\end{thm}
\begin{proof}
If $u$ is the classical solution to \eqref{linear CP with u0=0}, then $w(t,x)=u(t,x)-tu_1(x)$ is the classical solution to
\begin{align*}
\begin{cases}
w_{tt}-(1+t)^{2\ell} w_{xx}=t(1+t)^{2\ell}u''_1(x),& t>0,\,\, x\in \mathbb{R},\\
w(0,x)=0, & x\in \mathbb{R},\\
w_t(0,x)=0, & x\in \mathbb{R}.
\end{cases}
\end{align*}
Therefore, since we have an explicit representation formula for solutions to the inhomogeneous Cauchy problem related to \eqref{linear CP with u0=0} with vanishing data  we can directly derive a representation formula for $u$.

Indeed, using the same notations of the previous section we have
\begin{align*}
u(t,x)&=t u_1(x)+w(t,x)= t u_1(x) +c_\ell \int_0^t b(1+b)^{2\ell} \int_{A(b)-A(t)}^{A(t)-A(b)}  u''_1(x-y) E(t,y;b,0) \,dy \, db.
\end{align*}
Using  $u''_1(x-y)=\frac{\partial^2 u_1}{\partial y^2}(x-y)$ and applying  twice the integration by parts with respect to the variable $y$ we obtain
\begin{align*}
 \int_{A(b)-A(t)}^{A(t)-A(b)} &\frac{\partial^2 u_1}{\partial y^2}(x-y)E(t,y;b,0) \,dy
 \\ &= \Big[-u'_1(x-y)E(t,y;b,0)\Big]_{y=A(b)-A(t)}^{y=A(t)-A(b)}-\Big[u_1(x-y)\frac{\partial E}{\partial y}(t,y;b,0)\Big]_{y=A(b)-A(t)}^{y=A(t)-A(b)}\\ & \quad +\int_{A(b)-A(t)}^{A(t)-A(b)} u_1(x-y)\frac{\partial^2 E}{\partial y^2}(t,y;b,0) \,dy .
\end{align*}

Since the function $E(t,y;b,0)$ is even with respect to $y$ and, consequently, $\frac{\partial E}{\partial y}(t,y;b,0)$ is odd with respect to $y$, we find

\begin{align}
 \int_{A(b)-A(t)}^{A(t)-A(b)} &\frac{\partial^2 u_1}{\partial y^2}(x-y)E(t,y;b,0) \,dy \notag \\ &= \big(u'_1(x+A(t)-A(b))-u'_1(x-A(t)+A(b))\big)\,E(t,A(t)-A(b);b,0) \notag\\ & \quad -\big(u_1(x+A(t)-A(b))+u_1(x-A(t)+A(b))\big)\,\frac{\partial E}{\partial y}(t,y;b,0)\Big|_{y=A(t)-A(b)} \notag\\ & \quad +\int_{A(b)-A(t)}^{A(t)-A(b)} u_1(x-y)\frac{\partial^2 E}{\partial y^2}(t,y;b,0) \,dy .\label{y integral 1}
\end{align}

 Using the expressions  \eqref{E in y=A(t)-A(b)}  and \eqref{dE/dy in y=A(t)-A(b)} in \eqref{y integral 1} we obtain
\begin{align}
 \int_{A(b)-A(t)}^{A(t)-A(b)} & \frac{\partial^2 u_1}{\partial y^2}(x-y)E(t,y;b,0) \,dy  \notag  \\ &= 2^{-2\gamma} (\ell+1)^{2\gamma}\,(1+t)^{-\frac{\ell}{2}}(1+b)^{-\frac{\ell}{2}}  \big(u'_1(x+A(t)-A(b))-u'_1(x-A(t)+A(b))\big) \notag\\ & \quad  - 2^{-2\gamma-3} \ell(\ell+2) (\ell+1)^{2\gamma}\,(1+t)^{-\frac{3}{2}\ell-1}(1+b)^{-\frac{3}{2}\ell-1} \big(\phi(1+t)-\phi(1+b)\big)\notag\\ &\qquad \qquad \times \big(u_1(x+A(t)-A(b))+u_1(x-A(t)+A(b))\big) \notag \\ &\quad +\int_{A(b)-A(t)}^{A(t)-A(b)} u_1(x-y)\frac{\partial^2 E}{\partial y^2}(t,y;b,0) \,dy.\label{y integral}
\end{align}
Let us go back to the representation formula for $u=u(t,x)$. Using \eqref{y integral} we can write $u(t,x)$ as a sum of five terms in the following way:
\begin{align*}
u(t,x)&=t u_1(x) +c_\ell \int_0^t b(1+b)^{2\ell} \int_{A(b)-A(t)}^{A(t)-A(b)}  \frac{\partial^2 u_1}{\partial y^2}(x-y) E(t,y;b,0) \,dy \, db \\
&=t u_1(x)+ J_1+J_2+J_3+J_4,
\end{align*} where
\begin{align*}
J_1&:= \tfrac{1}{2}  (1+t)^{-\frac{\ell}{2}}\int_0^t b(1+b)^{\frac{3\ell}{2}}  \big(u'_1(x+A(t)-A(b))-u'_1(x-A(t)+A(b)) \big)\, db, \\
J_2&:= -\tfrac{\ell  (\ell+2)}{16(\ell+1)}  (1+t)^{-\frac{\ell}{2}}\int_0^t b(1+b)^{\frac{\ell}{2}-1}  \big(u_1(x+A(t)-A(b))+u_1(x-A(t)+A(b)) \big)\, db, \\
J_3&:= \tfrac{\ell  (\ell+2)}{16(\ell+1)} (1+t)^{-\frac{3\ell}{2}-1}\int_0^t b(1+b)^{\frac{3\ell}{2}}  \big(u_1(x+A(t)-A(b))+u_1(x-A(t)+A(b)) \big)\, db, \\
J_4&:= c_\ell \int_0^t b(1+b)^{2\ell}\int_{A(b)-A(t)}^{A(t)-A(b)} u_1(x-y)\frac{\partial^2 E}{\partial y^2}(t,y;b,0) \,dy \, db.
\end{align*}
Let us remark  that
\begin{equation}
\begin{split}
u'_1(x+A(t)-A(b))&=-(1+b)^{-\ell}\frac{\partial u_1}{\partial b}(x+A(t)-A(b)), \\
u'_1(x-A(t)+A(b))&=(1+b)^{-\ell}\frac{\partial u_1}{\partial b}(x-A(t)+A(b)).\label{relation u'_1 du_1/db}
\end{split}
\end{equation}

Therefore, we can rewrite $J_1$ as follows:
\begin{align*}
J_1&= -\tfrac{1}{2}  (1+t)^{-\frac{\ell}{2}}\int_0^t b(1+b)^{\frac{\ell}{2}}  \Big(\frac{\partial u_1}{\partial b}(x+A(t)-A(b))+\frac{\partial u_1}{\partial b}(x-A(t)+A(b)) \Big)\, db \\
 & =-t u_1(x)+\tfrac{1}{2}  (1+t)^{-\frac{\ell}{2}}\int_0^t (1+b)^{\frac{\ell}{2}}  \big(u_1(x+A(t)-A(b))+u_1(x-A(t)+A(b)) \big)\, db \\ & \quad+\tfrac{\ell}{4}  (1+t)^{-\frac{\ell}{2}}\int_0^t b(1+b)^{\frac{\ell}{2}-1}  \big(u_1(x+A(t)-A(b))+u_1(x-A(t)+A(b)) \big) db.
\end{align*}

 In particular, we see that thanks to this representation there is a cancelation of the term $t u_1(x)$ in the expression for $u(t,x)$ and that the second integral on the right-hand side is proportional to $J_2$.

Let us define the kernel
\begin{align*}
Q_1=Q_1(t,b)&=\tfrac{1}{2}(1+t)^{-\frac{\ell}{2}}(1+b)^{\frac{\ell}{2}} + \tfrac{\ell(3\ell+2)}{16(\ell+1)}  (1+t)^{-\frac{\ell}{2}} b(1+b)^{\frac{\ell}{2}-1} +\tfrac{\ell  (\ell+2)}{16(\ell+1)}  (1+t)^{-\frac{3\ell}{2}-1} b(1+b)^{\frac{3\ell}{2}}.
\end{align*}

Then, we can write $u(t,x)$ as
\begin{align}
u(t,x)&= \int_0^t Q_1(t,b) \big(u_1(x+A(t)-A(b))+u_1(x-A(t)+A(b)) \big)\, db +J_4.\label{J_1+...+J4}
\end{align}
Now we want to write the integral $J_4$ in a more suitable way. Firstly, we observe that $\frac{\partial^2 E}{\partial y^2}(t,y;b,0)$ is an even function in $y$, and then, by using Fubini's theorem, we arrive at
\begin{align*}
J_4&= c_\ell \int_0^t b(1+b)^{2\ell}\int_{0}^{A(t)-A(b)} \big(u_1(x-y)+u_1(x+y)\big)\frac{\partial^2 E}{\partial y^2}(t,y;b,0) \,dy \, db  \\ &= c_\ell \int_{0}^{A(t)} \big(u_1(x-y)+u_1(x+y)\big)\int_0^{A^{-1}(A(t)-y)} b(1+b)^{2\ell} \frac{\partial^2 E}{\partial y^2}(t,y;b,0) \,db \, dy.
\end{align*}


 On the other hand, $E(t,y;b,0)$ is symmetric with respect to the variables $t$ and $b$ (see \eqref{symmetry E t,b}).

Thus, using the fact that $E$ is a fundamental solution for the operator $\mathcal{T}$, we have in $D_+(t,0)$ the identity
\begin{align}
\frac{\partial^2 E}{\partial b^2}(b,y;t,0)=(1+b)^{2\ell} \frac{\partial^2 E}{\partial y^2}(b,y;t,0).\label{E is a fund sol property}
\end{align}

Consequently, a further integration by parts and the fundamental theorem of calculus provide
\begin{align}
J_4&= c_\ell \int_{0}^{A(t)}\big(u_1(x-y)+u_1(x+y)\big)\int_0^{A^{-1}(A(t)-y)}   b\, \frac{\partial^2 E}{\partial b^2}(b,y;t,0) \,db \, dy \notag\\
&=  c_\ell \int_{0}^{A(t)}\big(u_1(x-y)+u_1(x+y)\big)  \Big( b\, \frac{\partial E}{\partial b}(b,y;t,0) \Big)\Big|_{b=A^{-1}(A(t)-y)} \, dy \notag \\ &\quad +c_\ell \int_{0}^{A(t)}\big(u_1(x-y)+u_1(x+y)\big)\Big( E(t,y;0,0)-E(A^{-1}(A(t)-y),y;t,0) \Big) \, dy.\label{J_4}
\end{align}

Using the explicit expression of $E(b,y;t,0) $ and $\frac{\partial E}{\partial b}(b,y;t,0) $ for $b=A^{-1}(A(t)-y)$ (see \eqref{E in y=A^-1(A(t)-y)} and \eqref{dE/db for b=A^-1(A(t)-y)}) in \eqref{J_4} we can write $J_4$ as follows:
\begin{align*}
J_4&=   \int_{0}^{A(t)}\big(u_1(x-y)+u_1(x+y)\big)
\Big(\widetilde{Q}_1(t,y)+c_\ell E(t,y;0,0) \Big)\, dy,
\end{align*} where
\begin{align*}
\widetilde{Q}_1(t,y)&=-\tfrac{1}{4}\big(((1+t)^{\ell+1}-(\ell+1)y)^{\frac{1}{\ell+1}}-1\big)(\phi(1+t)-y)^{\gamma-1}  \phi(1+t)^{-\gamma-1}\big\{\gamma(2\phi(1+t)-y)+\gamma^2 y\big\}\\
&\quad - \tfrac{1}{2}(\ell+1)^{-\gamma} (1+t)^{-\frac{\ell}{2}}(\phi(1+t)-y)^{-\gamma}.
\end{align*}

Let us perform the change of variable $y=A(t)-A(b)$ for the first addend  in the last representation of $J_4$ we may conclude
\begin{align*}
J_4&=   \int_{0}^{t}\big(u_1(x-A(t)+A(b))+u_1(x+A(t)-A(b))\big)
\widetilde{Q}_1(t,A(t)-A(b))(1+b)^\ell db\\ &\quad +c_\ell \int_{0}^{A(t)}\big(u_1(x-y)+u_1(x+y)\big) E(t,y;0,0) \, dy.
\end{align*}

Since for $y=A(t)-A(b)$ the relations
\begin{align*}
b&=A^{-1}(A(t)-y)=((1+t)^{\ell+1}-(\ell+1)y)^{\frac{1}{\ell+1}}-1,
\end{align*} and \eqref{phi(1+b)=phi(1+t)-y} are fulfilled, then, we can write $\widetilde{Q}_1(t,A(t)-A(b))$ as follows:
\begin{align*}
\widetilde{Q}_1(t,A(t)-A(b))&=-\tfrac{1}{4}b(\phi(1+b))^{\gamma-1}  \phi(1+t)^{-\gamma-1}\big\{\gamma(\phi(1+t)+\phi(1+b))+\gamma^2 (\phi(1+t)-\phi(1+b))\big\}\\
&\quad - \tfrac{1}{2}(\ell+1)^{-\gamma} (1+t)^{-\frac{\ell}{2}}(\phi(1+b))^{-\gamma}\\
&= -\tfrac{1}{4} (\ell+1) (\gamma+\gamma^2) b(1+b)^{-\frac{\ell}{2}-1}(1+t)^{-\frac{\ell}{2}} -\tfrac{1}{4} (\ell+1) (\gamma-\gamma^2) b(1+b)^{\frac{\ell}{2}}(1+t)^{-\frac{3\ell}{2} -1}\\ & \quad -\tfrac{1}{2}(1+b)^{-\frac{\ell}{2}}(1+t)^{-\frac{\ell}{2}}=-Q_1(t,b)(1+b)^{-\ell}.
\end{align*}
Consequently, in \eqref{J_1+...+J4} there are cancelations between several terms and at the end, we obtain
\begin{align*}
u(t,x)=\int_{0}^{A(t)} \big(u_1(x-y)+u_1(x+y)\big)K_1(t,y) \, dy \, ,
\end{align*} where the kernel $K_1=K_1(t,y)$ is defined by \eqref{def K1(t,y)}.
\end{proof}
\begin{rmk}
In the last integral the kernel $K_1=K_1(t,y)$ is positive for $t\in [-A(t),A(t)]$.
\end{rmk}

%
%
%

\subsubsection{Homogeneous case with vanishing second data} \label{sssection u0}

\begin{thm}\label{thm repres formula 1d lin tric u_1=0}
The solution $u=u(t,x)$ of the Cauchy problem
\begin{align}\label{linear CP with u1=0}
\begin{cases}
u_{tt}-(1+t)^{2\ell} u_{xx}=0,& t>0,\,\, x\in \mathbb{R},\\
u(0,x)=u_0(x), & x\in \mathbb{R},\\
u_t(0,x)=0, & x\in \mathbb{R},
\end{cases}
\end{align}
with $u_0\in C^\infty_0 (\mathbb{R})$ can be represented as follows:
\begin{align*}
u(t,x)= \tfrac{1}{2}  (1+t)^{-\frac{\ell}{2}}\big(u_0(x+A(t))+u_0(x-A(t))\big)+\int_{0}^{A(t)} \big(u_0(x-y)+u_0(x+y)\big)K_0(t,y) \, dy,
\end{align*} where the kernel $K_0=K_0(t,y)$ is defined by
\begin{align}\label{definition K0(t,y)}
K_0(t,y)=-c_\ell \frac{\partial E}{\partial b}(t,y;b,0)\Big|_{b=0}.
\end{align}
\end{thm}

\begin{rmk}
Let us remark that the relations between the kernels $E$, $K_0$ and $K_1$ given by \eqref{def K1(t,y) via E} and \eqref{definition K0(t,y)} are formally identical (up to a multiplicative constant) to those in \cite[page 302]{YagGal09}.
\end{rmk}

\begin{proof}
If $u$ is the classical solution to \eqref{linear CP with u1=0}, then $w(t,x)=u(t,x)-u_0(x)$ is the classical solution to
\begin{align*}
\begin{cases}
w_{tt}-(1+t)^{2\ell} w_{xx}=(1+t)^{2\ell}u''_0(x),& t>0,\,\, x\in \mathbb{R},\\
w(0,x)=0, & x\in \mathbb{R},\\
w_t(0,x)=0, & x\in \mathbb{R}.
\end{cases}
\end{align*}  From this point we can proceed analogously as in the previous case.

Indeed, we have an explicit representation formula for solutions to the inhomogeneous Cauchy problem related to \eqref{linear CP with u1=0} with vanishing data, which provides a representation formula for $u(t,x)=w(t,x)+u_0(x)$.

Hence, using the notations of the previous sections we have
\begin{align*}
u(t,x)&= u_0(x)+w(t,x)=  u_0(x) +c_\ell \int_0^t (1+b)^{2\ell} \int_{A(b)-A(t)}^{A(t)-A(b)}  u''_0(x-y) E(t,y;b,0) \,dy \, db.
\end{align*}
Let us remark that the integral with respect to $y$ in the previous relation can be estimated exactly as we did in the case $u_0=0$. Thus, we arrive at
\begin{align}
& \int_{A(b)-A(t)}^{A(t)-A(b)} \frac{\partial^2 u_0}{\partial y^2}(x-y)E(t,y;b,0) \,dy \notag  \\ & \quad =  2^{-2\gamma} (\ell+1)^{2\gamma}\,(1+t)^{-\frac{\ell}{2}}(1+b)^{-\frac{\ell}{2}}  \big(u'_0(x+A(t)-A(b))-u'_0(x-A(t)+A(b))\big) \notag\\ & \qquad - 2^{-2\gamma-3} \ell(\ell+2) (\ell+1)^{2\gamma}\,(1+t)^{-\frac{3}{2}\ell-1}(1+b)^{-\frac{3}{2}\ell-1} \big(\phi(1+t)-\phi(1+b)\big)\notag\\ &\qquad \qquad \qquad \times \big(u_0(x+A(t)-A(b))+u_0(x-A(t)+A(b))\big) \notag \\ &  \qquad + \int_{A(b)-A(t)}^{A(t)-A(b)} u_0(x-y)\frac{\partial^2 E}{\partial y^2}(t,y;b,0) \,dy .\label{y integral, u0}
\end{align}
Using \eqref{y integral, u0} we can write $u(t,x)$ as a sum of five terms in the following way:
\begin{align*}
u(t,x)&= u_0(x) +c_\ell \int_0^t (1+b)^{2\ell} \int_{A(b)-A(t)}^{A(t)-A(b)}  \frac{\partial^2 u_0}{\partial y^2}(x-y) E(t,y;b,0) \,dy \, db \\
&= u_0(x)+ I_1+I_2+I_3+I_4,
\end{align*} where
\begin{align*}
I_1&:= \tfrac{1}{2}  (1+t)^{-\frac{\ell}{2}}\int_0^t (1+b)^{\frac{3\ell}{2}}  \big(u'_0(x+A(t)-A(b))-u'_0(x-A(t)+A(b)) \big)\, db, \\
I_2&:= -\tfrac{\ell  (\ell+2)}{16(\ell+1)}  (1+t)^{-\frac{\ell}{2}}\int_0^t (1+b)^{\frac{\ell}{2}-1}  \big(u_0(x+A(t)-A(b))+u_0(x-A(t)+A(b)) \big)\, db, \\
I_3&:= \tfrac{\ell  (\ell+2)}{16(\ell+1)} (1+t)^{-\frac{3\ell}{2}-1}\int_0^t (1+b)^{\frac{3\ell}{2}}  \big(u_0(x+A(t)-A(b))+u_0(x-A(t)+A(b)) \big)\, db, \\
I_4&:= c_\ell \int_0^t (1+b)^{2\ell}\int_{A(b)-A(t)}^{A(t)-A(b)} u_0(x-y)\frac{\partial^2 E}{\partial y^2}(t,y;b,0) \,dy \, db.
\end{align*}

Using analogous relations as in \eqref{relation u'_1 du_1/db} for $u_0$, we can rewrite $I_1$ as follows:
\begin{align*}
I_1&= -\tfrac{1}{2}  (1+t)^{-\frac{\ell}{2}}\int_0^t (1+b)^{\frac{\ell}{2}}  \Big(\frac{\partial u_0}{\partial b}(x+A(t)-A(b))+\frac{\partial u_0}{\partial b}(x-A(t)+A(b)) \Big)\, db \\
 & =- u_0(x)+\tfrac{1}{2}  (1+t)^{-\frac{\ell}{2}}\big(u_0(x+A(t))+u_0(x-A(t))\big)\\ & \quad+\tfrac{\ell}{4}  (1+t)^{-\frac{\ell}{2}}\int_0^t (1+b)^{\frac{\ell}{2}-1}  \big(u_0(x+A(t)-A(b))+u_0(x-A(t)+A(b)) \big)\, db.
\end{align*} In particular, we see that thanks to this representation there is a cancelation of the term $ u_0(x)$ in the expression for $u(t,x)$ and that the last integral is proportional to $I_2$.

Let us define the kernel
\begin{align*}
Q_0=Q_0(t,b)&= \tfrac{\ell(3\ell+2)}{16(\ell+1)}  (1+t)^{-\frac{\ell}{2}} (1+b)^{\frac{\ell}{2}-1} +\tfrac{\ell  (\ell+2)}{16(\ell+1)}  (1+t)^{-\frac{3\ell}{2}-1} (1+b)^{\frac{3\ell}{2}}.
\end{align*}

Then, we may rewrite $u(t,x)$ as follows:
\begin{align}
u(t,x)&= \tfrac{1}{2}  (1+t)^{-\frac{\ell}{2}}\big(u_0(x+A(t))+u_0(x-A(t))\big)\notag\\ &\quad +\int_0^t Q_0(t,b) \big(u_0(x+A(t)-A(b))+u_0(x-A(t)+A(b)) \big)\, db +I_4. \label{K_1+...+K4}
\end{align}

As we did for the term $J_4$ in the previous case, now we are going to rewrite $I_4$ in a suitable way, in order to have a cancelation of several terms between $I_4$ and the last integral with kernel $Q_0(t,b)$ in \eqref{K_1+...+K4}.

 Since $\frac{\partial^2 E}{\partial y^2}(t,y;b,0)$ is an even function in $y$, because of Fubini's theorem, we arrive at
\begin{align*}
I_4&= c_\ell \int_0^t (1+b)^{2\ell}\int_{0}^{A(t)-A(b)} \big(u_0(x-y)+u_0(x+y)\big)\frac{\partial^2 E}{\partial y^2}(t,y;b,0) \,dy \, db  \\ &= c_\ell \int_{0}^{A(t)} \big(u_0(x-y)+u_0(x+y)\big)\int_0^{A^{-1}(A(t)-y)} (1+b)^{2\ell} \frac{\partial^2 E}{\partial y^2}(t,y;b,0) \,db \, dy  \,.
\end{align*}

By using \eqref{E is a fund sol property} and the fundamental theorem of calculus we get
\begin{align}
I_4 
&=  c_\ell \int_{0}^{A(t)}\big(u_0(x-y)+u_0(x+y)\big) \Big[  \frac{\partial E}{\partial b}(b,y;t,0) \Big]_{b=0}^{b=A^{-1}(A(t)-y)}\ dy.\label{K_4}
\end{align}

Substituting \eqref{dE/db for b=A^-1(A(t)-y)}  in \eqref{K_4}, we can write $I_4$ as follows:
\begin{align*}
I_4&=   \int_{0}^{A(t)}\big(u_0(x-y)+u_0(x+y)\big)
\Big(\widetilde{Q}_0(t,y)-c_\ell \frac{\partial E}{\partial b}(t,y;b,0)\Big|_{b=0} \Big)\, dy,
\end{align*} where
\begin{align*}
\widetilde{Q}_0(t,y)&=-\tfrac{1}{4}(\phi(1+t)-y)^{\gamma-1}  \phi(1+t)^{-\gamma-1}\big[\gamma(2\phi(1+t)-y)+\gamma^2 y\big].
\end{align*}

The change of variables $y=A(t)-A(b)$ in the first addend in the previous representation for $I_4$ implies
\begin{align*}
I_4&=   \int_{0}^{t}\big(u_0(x-A(t)+A(b))+u_0(x+A(t)-A(b))\big)
\widetilde{Q}_0(t,A(t)-A(b))(1+b)^\ell db\\ &\quad -c_\ell \int_{0}^{A(t)}\big(u_0(x-y)+u_0(x+y)\big) \frac{\partial E}{\partial b}(t,y;b,0)\Big|_{b=0} \, dy \, .
\end{align*}

Since for $y=A(t)-A(b)$ the condition \eqref{phi(1+b)=phi(1+t)-y} is satisfied, then we can express $\widetilde{Q}_0(t,A(t)-A(b))$ as follows:
\begin{align*}
\widetilde{Q}_0(t,A(t)-A(b))&=-\tfrac{1}{4}(\phi(1+b))^{\gamma-1}  \phi(1+t)^{-\gamma-1}\big[\gamma(\phi(1+t)+\phi(1+b))+\gamma^2 (\phi(1+t)-\phi(1+b))\big]\\
&= -\tfrac{1}{4} (\ell+1) (\gamma+\gamma^2) (1+b)^{-\frac{\ell}{2}-1}(1+t)^{-\frac{\ell}{2}} -\tfrac{1}{4} (\ell+1) (\gamma-\gamma^2) (1+b)^{\frac{\ell}{2}}(1+t)^{-\frac{3\ell}{2} -1}\\ & =-Q_0(t,b)(1+b)^{-\ell}.
\end{align*}

Consequently, from \eqref{K_1+...+K4} it follows
\begin{align*}
u(t,x)&=\tfrac{1}{2}  (1+t)^{-\frac{\ell}{2}}\big(u_0(x+A(t))+u_0(x-A(t))\big) +\int_{0}^{A(t)} \big(u_0(x-y)+u_0(x+y)\big)K_0(t,y) \, dy \, ,
\end{align*} where the kernel $K_0=K_0(t,y)$ is defined by \eqref{definition K0(t,y)}.
\end{proof}
\begin{rmk}
The kernel $K_0=K_0(t,y)$ is nonnegative on the domain of integration.
Indeed, the Gauss hypergeometric functions $F(\gamma,\gamma;1;z)$ and $F(\gamma+1,\gamma+1;2;z)$ are nonnegative for $b\in [0,t]$ and $y\in [A(b)-A(t),A(t)-A(b)]$ and for $z=\frac{(\phi(1+b)-\phi(1+t))^2-y^2}{(\phi(1+b)+\phi(1+t))^2-y^2}$ .

 Furthermore, for $b\in [0,t]$ and $y\in [A(b)-A(t),A(t)-A(b)]$
\begin{align*}
y^2+\phi(1+b)^2-\phi(1+t)^2&\leq (\phi(1+t)-\phi(1+b))^2+\phi(1+b)^2-\phi(1+t)^2 \\ &=2\phi(1+b)(\phi(1+b)-\phi(1+t))\leq 0.
\end{align*}
Hence, from \eqref{dE/db} it follows that $\frac{\partial E}{\partial b}(t,y;b,0)\leq 0$ for $b\in [0,t]$ and $y\in [A(b)-A(t),A(t)-A(b)]$.
In particular, we find for $b=0$ that $K_0(t,y)\geq 0$ for $y\in [-A(t),A(t)]$.
\end{rmk}
\subsection{Representation formula}\label{Section Summary repres formula}
Summarizing, with the same notations as before,  we derived in the previous sections the representation formula
\begin{align}
u(t,x)&=\tfrac{1}{2}  (1+t)^{-\frac{\ell}{2}}\big(u_0(x+A(t))+u_0(x-A(t))\big)+\int_{0}^{A(t)} \big(u_0(x-y)+u_0(x+y)\big)K_0(t,y) \, dy \notag\\ & \quad + \int_{0}^{A(t)} \big(u_1(x-y)+u_1(x+y)\big)K_1(t,y) \, dy \notag\\ & \quad + c_\ell \int_0^t  \int_{0}^{\phi(1+t)-\phi(1+b)}   [ f(b,x-y)+f(b,x+y) ] \, E(t,y;b,0)\, dy\, db \label{representation formula tric 1d}
\end{align}
for classical solutions (even for weak solutions if all integrals are defined) of the Cauchy problem \eqref{Tric type CP 1d}.

In particular from the previous representation formula it is clear that the \emph{domain of dependence} for the solution $u$ in the point $(t_0,x_0)$, denoted by $\Omega(t_0,x_0)$, is the intersection of the backward characteristic cone with the upper half-plane, that is,
\begin{align*}
\Omega(t_0,x_0)=\big\{(t,x)\in \mathbb{R}^2: 0\leq t < t_0, |x-x_0| < \phi(1+t_0)-\phi(1+t)\big\}.
\end{align*}
\begin{rmk} In the case $\ell=0$, \eqref{representation formula tric 1d} coincides with the well-known D'Alembert's formula for the free wave equation in 1d.
\end{rmk}
\begin{rmk}\label{rmk nonnegativity hom problem} Because of the nonnegativity of the kernels $K_0=K_0(t,y)$ and $K_1=K_1(t,y)$ for $t\in [-A(t),A(t)]$, after assuming additionally nonnegative initial data $u_0,u_1$ we may conclude that classical solutions (even weak solutions) of the linear homogeneous Cauchy problem \eqref{Tric type CP 1d lin hom}  are everywhere (almost everywhere) nonnegative. In particular, under the assumption of nonnegativity of the data $u_0,u_1$, we can estimate from below the solution of the corresponding inhomogeneous Cauchy problem \eqref{Tric type CP 1d} by the solution of the inhomogeneous Cauchy problem with the same source but with vanishing initial data.
\end{rmk}

\section{Critical case} \label{Section Blowup critical case}
\setcounter{equation}{0}

Let us prove the blow-up result for \eqref{CP semi Tric eq} in the critical case, that is, when we employ Kato's lemma for the case in which the exponent $a$ in \eqref{condition on  of F Kato} satisfies $a=\frac{q-2}{p-1}$ and $k_0=k_0(k_1) >0$ is sufficiently large. This condition corresponds to the requirements \eqref{condition on p critical p1} and \eqref{condition on p critical p0} in the statement of Theorem \ref{Main theo}.

\begin{proof}[Proof of Theorem \ref{Main theo} in the critical case]

In this section we complete the proof of Theorem \ref{Main theo} in the critical case $p=p_{\NE}(n;\ell,k)$.
Let us start with the first case $p=p_1(n;\ell,k)$, or in other terms, when
$$-\frac{k+2}{p-1}+(\ell+1)n=1.$$
From \eqref{condition on G Kato} it follows
\begin{align*}
\mathscr{G}(t)\gtrsim (R+t).
\end{align*}

Therefore, by \eqref{condition on G'' Kato} we get
\begin{align*}
\frac{d^2 \mathscr{G}}{dt^2}(t)\gtrsim (R+t)^{k-(\ell+1)n(p-1)}(\mathscr{G}(t))^p = (R+t)^{-p-1}(\mathscr{G}(t))^p \gtrsim (R+t)^{-1}.
\end{align*}

Now it is opportune to mention explicitly the multiplicative constant, so let us consider $C>0$ such that $$\frac{d^2 \mathscr{G}}{dt^2}(t)\geq C (R+t)^{-1}.$$
A first integration leads to
\begin{align*}
\frac{d \mathscr{G}}{dt}(t)\geq \frac{d \mathscr{G}}{dt}(0)+C\log\big(\tfrac{R+t}{R}\big).
\end{align*}
 Integrating a second time, we find for large $t$ the estimate
\begin{align*}
\mathscr{G}(t)&\geq \mathscr{G}(0)+\frac{d \mathscr{G}}{dt}(0)\,t+C \int_0^t \log\big(\tfrac{R+s}{R}\big)ds \\ &=\mathscr{G}(0)+\frac{d \mathscr{G}}{dt}(0)\,t+C (R+t)\log(R+t)-C((R+t)\log R +t) \gtrsim (R+t)\log\big(\tfrac{R+t}{R}\big).
\end{align*}
According to Remark \ref{remark log term} the solution of our semi-linear Cauchy problem blows up also for $p=p_1(n;\ell,k)=p_{\NE}(n;\ell,k)$.

Using the representation formula from Section \ref{Section repres formula Tric eq} we shall, finally, consider the critical case $p=p_0(n;\ell,k)$, that is, when it holds the equality $$k-\frac{\ell p}{2}-(n-1)\Big(\frac{p}{2}-1\Big)(\ell+1)+2=-\frac{k+2}{p-1}+(\ell+1)n$$ instead of \eqref{1st condition p blow up Tric} in Theorem \ref{Main theo}.

 In fact for this case, as we already explained in Remark \ref{remark log term}, in order to prove \eqref{condition on G Kato} for an arbitrarily large multiplicative constant, we shall follow the approach of \cite{Yor06} (in this paper blow-up is studied in the critical case for the classical wave equation in spatial dimension $n\geq 4$) and of \cite{HeWittYinA17} (the critical case is considered for the Tricomi equation in dimensions $n\geq 2$). In particular, the key idea to prove the case $n=2$ is taken from \cite{HeWittYinA17}.



  Without loss of generality we may assume that $u=u(t,r)$ is a radial function. Indeed, if $u$ is not radial, considering the spherical mean
 \begin{align*}
 \widetilde{u}(t,r)=\frac{1}{\omega_n} \int_{\mathbb{S}^{n-1}}u(t,r\omega) d\sigma_\omega,
 \end{align*} of $u$, where $\omega_n$ is the $(n-1)$-dimensional measure of the unit sphere $\mathbb{S}^{n-1}$, then we have that $\widetilde{u}$ satisfies
 \begin{align*}
 \widetilde{u}_{tt}-(1+t)^{2\ell}\Delta\widetilde{u}=(\ell+1)^2(1+t)^k\widetilde{|u|^p}\geq(\ell+1)^2(1+t)^k|\widetilde{u}|^p.
\end{align*}

Let us consider the Radon transform
\begin{align*}
\mathscr{R}[u](t,\rho)=\int_{x\cdot \xi= \rho} u(t,x)\,d\sigma_x=\int_{x\cdot\xi=0}u(t,\rho\xi +x)\,d\sigma_x
\end{align*} of $u$, where $\rho\in\mathbb{R}$, $\xi\in\mathbb{R}^n$ is a given unitary vector and $d\sigma_x$ is the Lebesgue measure on the corresponding hyperplanes.  Since $u=u(t,r)$ is a radial function, $\mathscr{R}[u]$ does not depend on $\xi$ and
\begin{align*}
\mathscr{R}[u](t,\rho)= \omega_{n-1} \int_{|\rho|}^\infty u(t,r)(r^2-\rho^2)^{\frac{n-3}{2}}rdr
\end{align*} (cf. \cite[Formula (2.10) at page 368]{Yor06}).

The function $\mathscr{R}[u]$ solves the equation
\begin{align*}
\partial_t^2 \mathscr{R}[u]-(1+t)^{2\ell}\partial^2_\rho \mathscr{R}[u]=(\ell+1)^2(1+t)^k\mathscr{R}[|u|^p](t,\rho),
\end{align*} since $\mathscr{R}$ is linear and $\mathscr{R}[\Delta u]=\partial^2_\rho \mathscr{R}[u]$.

Assuming nonnegative data $u_0,u_1$ the solution of the homogeneous linear Cauchy problem
\begin{align*}
\begin{cases} w_{tt}-(1+t)^{2\ell} w_{\rho\rho}=0, \\
w(0,\rho)=\mathscr{R}[u_0](\rho),\\
w_\rho(0,\rho)=\mathscr{R}[u_1](\rho),\end{cases}
\end{align*} is always nonnegative. Hence, the representation formula \eqref{representation formula tric 1d} and Remark \ref{rmk nonnegativity hom problem} yield
\begin{align*}
\mathscr{R}[u](t,\rho)\geq c_\ell (\ell+1)^2 \int_0^t  (1+b)^k\! \int_0^{A(t)-A(b)} \big(\mathscr{R}[|u|^p](b,\rho-\sigma)+\mathscr{R}[|u|^p](b,\rho+\sigma)\big)E(t,\sigma;b,0)\,d\sigma db.
\end{align*}

Furthermore, the kernel function $E(t,\sigma;b,0)$ is positive on the domain of integration. Hence, we can integrate over a smaller domain and we obtain still a lower bound for $\mathscr{R}[u](t,\rho)$. Performing a change of variables in the previous integral we obtain
\begin{align*}
\mathscr{R}[u](t,\rho)\geq c_\ell (\ell+1)^2 \int_0^t  (1+b)^k \int_{\rho-(A(t)-A(b))}^{\rho+A(t)-A(b)} \mathscr{R}[|u|^p](b,\sigma)E(t,\rho-\sigma;b,0)\,d\sigma db.
\end{align*}

Now we observe that $$\supp \mathscr{R}[|u|^p](b,\cdot)\subset [-R-A(b),R+A(b)],$$ where $R>0$ is chosen in such a way that $\supp u_0, \supp u_1 \subset B_R$.

Indeed, thanks the finite speed of propagation property we know that $\supp u(b,\cdot)\subset B_{R+A(b)}$.

Therefore, if $|\sigma|> R+A(b)$ we have for any vector $x$ which is orthogonal to the unit vector $\xi$ the relation
\begin{align*}
\mathscr{R}[|u|^p](b,\sigma)=\int_{x\cdot \xi =0} |u(b,\sigma \xi+x)|^p d\sigma_x =0,
\end{align*} being $|\sigma \xi +x|\geq |\sigma|\geq R+A(b)$.

Let $t\geq 0$ such that $A(t)>R+1$ and $\rho\in (0, A(t)-R-1)$ be fixed (the condition on $t$ guarantees the nonemptiness of the interval for $\rho$).

 Since for these fixed values of $t$ and $\rho$ it results $1<A(t)-R-\rho < A(t)$, due to the monotonicity of $A$ we can find   $b_0\in (0,t) $ such that $2A(b_0)=A(t)-R-\rho $. Then for $0\leq b\leq b_0$ we have
\begin{align} \label{A(b) upper bound}
A(b)\leq A(b_0)= \tfrac{1}{2}(A(t)-R-\rho).
\end{align} Hence, it holds
\begin{align*}
\rho + A(t) -A(b) \geq R+A(b) \qquad \mbox{and} \qquad \rho -( A(t) -A(b)) \leq -(R+A(b)).
\end{align*}

Therefore,
\begin{align*}
\mathscr{R}[u](t,\rho)&\geq c_\ell (\ell+1)^2 \int_0^{b_0}  (1+b)^k \int_{-(R+A(b))}^{R+A(b)} \mathscr{R}[|u|^p](b,\sigma)E(t,\rho-\sigma;b,0)\,d\sigma db \\ &=c_\ell (\ell+1)^2 \int_0^{b_0}  (1+b)^k \int_{-\infty}^{+\infty} \mathscr{R}[|u|^p](b,\sigma)E(t,\rho-\sigma;b,0)\,d\sigma db,
\end{align*} because of the support property for $\mathscr{R}[|u|^p](b,\cdot)$.

Let us estimate from below the kernel
\begin{align*}
E(t,\rho-\sigma;b,0)=\big((\phi(1+t)+\phi(1+b))^2-(\rho-\sigma)^2\big)^{-\gamma} F\Big(\gamma,\gamma;1; \frac{(\phi(1+t)-\phi(1+b))^2-(\rho-\sigma)^2}{(\phi(1+t)+\phi(1+b))^2-(\rho-\sigma)^2}\Big)
\end{align*} on the domain of integration $\{(b,\sigma)\in \mathbb{R}^2: 0\leq b\leq b_0, \, |\sigma |\leq R+A(b)\}$. Since the argument of the Gauss hypergeometric function is always an element of the interval $[0,1)$ and the parameters $(\gamma,\gamma;1)$ are all nonnegative we may estimate
\begin{align*}
E(t,\rho-\sigma;b,0)&\geq ((\phi(1+t)+\phi(1+b))^2-(\rho-\sigma)^2)^{-\gamma} F(\gamma,\gamma;1;0)
\\& \gtrsim \phi (1+t)^{-\gamma} (\phi(1+t)-\rho)^{-\gamma}\simeq (1+t)^{-\frac{\ell}{2}} (\phi(1+t)-\rho)^{-\gamma}.
\end{align*}
Here we used the inequalities
\begin{align*}
\phi(1+t)+\phi(1+b)-\rho+\sigma & 
\leq 2\,(\phi(1+t)-\rho),\\
\phi(1+t)+\phi(1+b)+\rho-\sigma & 
 \leq 2\,\phi(1+t),
\end{align*} which are consequences  of \eqref{A(b) upper bound} and of the fact that $|\sigma|\leq R+A(b)$.

So, we get

\begin{align*}
\mathscr{R}[u](t,\rho)\geq c_\ell (1+t)^{-\frac{\ell}{2}} (\phi(1+t)-\rho)^{-\gamma} \int_0^t  (\ell+1)^2  (1+b)^k   \int_{\rho-(A(t)-A(b))}^{\rho+A(t)-A(b)} \mathscr{R}[|u|^p](b,\sigma)\,d\sigma db.
\end{align*}

 Consequently, because of the property for the support of $\mathscr{R}[|u|^p](b,\cdot)$, by using \eqref{condition on G'' Kato} we have
\begin{align*}
\mathscr{R}[u](t,\rho)&\geq c_\ell  (1+t)^{-\frac{\ell}{2}} (\phi(1+t)-\rho)^{-\gamma} \int_0^{b_0} (\ell+1)^2 (1+b)^k  \int_{-\infty}^{+\infty} \mathscr{R}[|u|^p](b,\sigma)\,d\sigma db \\
&  = c_\ell  (1+t)^{-\frac{\ell}{2}} (\phi(1+t)-\rho)^{-\gamma} \int_0^{b_0}  (\ell+1)^2(1+b)^k  \int_{\mathbb{R}^n} |u(b,x)|^p dx db\\
 &  = c_\ell  (1+t)^{-\frac{\ell}{2}} (\phi(1+t)-\rho)^{-\gamma}\int_0^{b_0} \frac{d^2 \mathscr{G}}{db^2}(b) db\\
 & \gtrsim  (1+t)^{-\frac{\ell}{2}} (\phi(1+t)-\rho)^{-\gamma} \int_0^{b_0} (R+b)^{k-\frac{\ell}{2}p-(n-1)(\frac{p}{2}-1)(\ell+1)} db.
\end{align*}
We recall that in the present case we deal with we assume $p=p_0(n;\ell,k)>p_1(n;\ell,k)$. But this is equivalent to require that the exponent of $(R+b)$ in the last integral is greater than $-1$.
Thus
\begin{align}
\mathscr{R}[u](t,\rho)&
 \gtrsim (1+t)^{-\frac{\ell}{2}} (\phi(1+t)-\rho)^{-\gamma}(1+b_0)^{k+1-\frac{\ell}{2}p-(n-1)(\frac{p}{2}-1)(\ell+1)} \notag \\
&\gtrsim (1+t)^{-\frac{\ell}{2}} (\phi(1+t)-\rho)^{-\gamma}(A(t)-R-\rho)^{\frac{k+1}{\ell+1}-\gamma p-(n-1)(\frac{p}{2}-1)}.\label{lower bound for the radon transf of u}
\end{align}

Let us introduce the operator $T:f\in L^p(\mathbb{R})\to T(f) \in L^p(\mathbb{R})$, where $T(f)$  is defined by
\begin{align*}
T(f)(\tau )=\frac{1}{|A(t)-\tau +R|^{\frac{n-1}{2}}} \int^{A(t)+R}_\tau f(r)|r-\tau|^{\frac{n-3}{2}}dr  \qquad \mbox{for any} \,\, \tau \in\mathbb{R}.
\end{align*}

In the case $n\geq 3$ the operator $T$ is a bounded operator on $L^p(\mathbb{R}^n)$. By triangular inequality we get
\begin{align*}
|T(f)(\tau)|&\leq \frac{1}{|A(t)-\tau +R|}\Big| \int^{A(t)+R}_\tau |f(r)|dr\Big|\leq \frac{1}{|A(t)-\tau +R|}\Big| \int^{A(t)+R}_{2\tau-(A(t)+R)} |f(r)|dr\Big| \leq 2 M(f)(\tau),
\end{align*} where $M(f)$ denotes the maximal function of $f$ (see Section \ref{Section maximal function}).

Being $p\in (1,\infty)$ it follows for $n\geq 3$ that
\begin{align*}
\| T(f)\|_{L^p(\mathbb{R})}\leq 2\| M(f)\|_{L^p(\mathbb{R})}\lesssim \| f\|_{L^p(\mathbb{R})}.
\end{align*}
Let us investigate the case $n=2$. In order to prove that $T$ is a bounded operator on $L^p(\mathbb{R})$ also in this case,
we will first derive $L^\infty(\mathbb{R})-L^\infty(\mathbb{R})$ and $L^1(\mathbb{R})-L^{1,\infty}(\mathbb{R})$ estimates (here $L^{1,\infty}(\mathbb{R})$ denotes the weak $L^1(\mathbb{R})$ space, cf. Section \ref{Section weak Lp}). Hence, by using Marcinkiewicz interpolation theorem (cf. Proposition \ref{Marcinkiewicz inter thrm}) we will reach the desired $L^p(\mathbb{R})-L^p(\mathbb{R})$ estimate.
In the case $n=2$ the operator $T$ is
\begin{align*}
T(f)(\tau )=|A(t)-\tau +R|^{-\frac{1}{2}} \int^{A(t)+R}_\tau f(r)|r-\tau|^{-\frac{1}{2}}dr  \qquad \mbox{for any} \,\, \tau \in\mathbb{R} \,\, \mbox{ and any} \,\, f\in L^p(\mathbb{R}).
\end{align*}

Let us begin with the $L^\infty(\mathbb{R})-L^\infty(\mathbb{R})$ estimate. For $\tau \leq A(t)+R$ we have
\begin{align}
|T(f)(\tau )|
&\leq (A(t)-\tau +R)^{-\frac{1}{2}} \|f\|_{L^\infty(\mathbb{R})} \int^{A(t)+R}_\tau (r-\tau)^{-\frac{1}{2}}dr =2  \|f\|_{L^\infty(\mathbb{R})}. \label{L^inf-L^inf est}
\end{align} When $\tau \geq A(t)+R$, in an analogous way we get the same estimate as in \eqref{L^inf-L^inf est}. Summarizing, $$\| T(f)\|_{L^\infty(\mathbb{R})}\leq 2 \|f\|_{L^\infty(\mathbb{R})}.$$

We prove now the $L^1(\mathbb{R})-L^{1,\infty}(\mathbb{R})$ estimate. Let $f\in L^1(\mathbb{R})$.
 We can write $T(f)$ as  product of the following two functions:
\begin{align*}
g(\tau)&= |A(t)-\tau +R|^{-\frac{1}{2}} \,, \qquad
h(\tau)=\int^{A(t)+R}_\tau f(r)|r-\tau|^{-\frac{1}{2}}dr.
\end{align*}

Now we want to show that $g,h\in L^{2,\infty}(\mathbb{R})$. By definition
\begin{align*}
\| g\|^2_{L^{2,\infty}(\mathbb{R})}&=\sup_{\alpha >0} \alpha^2 \meas\big(\{\tau \in \mathbb{R}: |g(\tau)|>\alpha\}\big)=2. 
\end{align*}

On the other hand
\begin{align*}
|h(\tau)|\leq \int_\mathbb{R} |f(r)| \,|r-\tau|^{-\frac{1}{2}}dr = (|f|\ast |r|^{-\frac{1}{2}}) (\tau),
\end{align*} and then by Young's inequality for weak type spaces (cf. Proposition \ref{Young ineq weak spac}) we obtain
\begin{align*}
 \| h\|_{L^{2,\infty}(\mathbb{R})}&\leq \| |f|\ast |r|^{-\frac{1}{2}}\|_{L^{2,\infty}(\mathbb{R})} \lesssim \|f\|_{L^1(\mathbb{R})}\,\|\,  |r|^{-\frac{1}{2}}\|_{L^{2,\infty}(\mathbb{R})} = 2^{\frac{1}{2}}\|f\|_{L^1(\mathbb{R})}.
\end{align*}

Consequently, by H\"{o}lder's inequality for weak type spaces (see also Proposition \ref{Holder ineq weak Lp}), we obtain
\begin{align}
\|T(f)\|_{L^{1,\infty}(\mathbb{R})}\leq 2\, \| g\|_{L^{2,\infty}(\mathbb{R})} \| h\|_{L^{2,\infty}(\mathbb{R})} \lesssim \, \|f\|_{L^1(\mathbb{R})}. \label{L^1-L^1,inf est}
\end{align}

Combining \eqref{L^inf-L^inf est} and  \eqref{L^1-L^1,inf est} and using Marcinkiewicz interpolation theorem, we arrive at
\begin{align*}
\| T(f)\|_{L^{p}(\mathbb{R})}  \leq C_p \|f\|_{L^p(\mathbb{R})},
\end{align*} where $C_p$  depends on $p$.

Now we can use the boundedness of $T$ to estimate $\|T(f)\|_{L^p(\mathbb{R})}$ for
\begin{align*}
f(t,r)=\begin{cases} |u(t,r)|r^{\frac{n-1}{p}} & \mbox{if} \,\, r\geq 0 , \\ 0 & \mbox{elsewhere}.  \end{cases}
\end{align*}
So,
\begin{align*}
\int_0^{A(t)-R-1}|T(f)(t,\rho)|^p d\rho & 
\lesssim \int_\mathbb{R} |f(t,\rho)|^p d\rho =\int_0^\infty |u(t,\rho)|^p \rho^{n-1} d\rho = \frac{1}{\omega_n}\int_{\mathbb{R}^n}|u(t,x)|^p dx.
\end{align*} Hence,
\begin{align*}
\int_{\mathbb{R}^n}|u(t,x)|^p dx &\gtrsim \int_0^{A(t)-R-1}|T(f)(t,\rho)|^p d\rho\\ & = \int_0^{A(t)-R-1} (A(t)-\rho+R)^{-\frac{n-1}{2}\, p} \Big(\int_\rho ^{A(t)+R}|u(t,r)|r^{\frac{n-1}{p}}(r-\rho)^{\frac{n-3}{2}} dr\Big)^pd\rho .
\end{align*}

For $\rho \leq r \leq A(t)+R$ it holds
\begin{align*}
 r^\frac{n-1}{p}&\geq r^\frac{n-1}{2}\rho^{(n-1)(\frac{1}{p}-\frac{1}{2})} \qquad \qquad \qquad\ \mbox{if} \,\, \,\, p \in (1,2],  \\ r^\frac{n-1}{p}&\geq r^\frac{n-1}{2}(A(t)-R-1)^{(n-1)(\frac{1}{p}-\frac{1}{2})} \qquad \,\mbox{if} \,\,\,\, p \geq 2.
\end{align*}
 Let us proceed with the computations in the case $p\leq 2$, since the computations are more straightforward in the case $p\geq 2$.
We get
\begin{align*}
\int_{\mathbb{R}^n}|u(t,x)|^p dx \gtrsim \int_0^{A(t)-R-1} (A(t)-\rho+R)^{-\frac{n-1}{2}\, p} \Big(\int_\rho ^{A(t)+R}|u(t,r)|r^{\frac{n-1}{2}}(r-\rho)^{\frac{n-3}{2}} dr\Big)^p \rho^{(n-1)(1-\frac{p}{2})} d\rho.
\end{align*}

We recognize that
\begin{align*}
\mathscr{R}[|u|](t,\rho)&= \omega_{n-1} \int_{|\rho|}^{A(t)+R}|u(t,\rho)|(r^2-\rho^2)^{\frac{n-3}{2}}rdr =\omega_{n-1} \int_{|\rho|}^{A(t)+R}|u(t,\rho)|(r-\rho)^{\frac{n-3}{2}}(r+\rho)^{\frac{n-3}{2}}rdr \\ &\leq 2^{\frac{n-3}{2}} \omega_{n-1} \int_{|\rho|}^{A(t)+R}|u(t,\rho)|(r-\rho)^{\frac{n-3}{2}}r^{\frac{n-1}{2}}dr.
\end{align*}

Therefore, since $0\leq \mathscr{R}[u](t,\rho)\leq \mathscr{R}[|u|](t,\rho)$,  
we find
\begin{align*}
\int_{\mathbb{R}^n}|u(t,x)|^p dx &\gtrsim \int_0^{A(t)-R-1} (A(t)-\rho+R)^{-\frac{n-1}{2}\, p} (\mathscr{R}[u](t,\rho))^p \rho^{(n-1)(1-\frac{p}{2})} d\rho \\  & \gtrsim (1+t)^{-\frac{\ell}{2} p}\int_0^{A(t)-R-1} (A(t)-\rho+R)^{-\frac{n-1}{2}\, p} (A(t)-\rho+\phi(1))^{-\gamma p}\\ & \qquad \qquad \qquad \qquad \qquad \times (A(t)-R-\rho)^{\big(\frac{k+1}{\ell+1}-\gamma p -(n-1)(\frac{p}{2}-1)\big)p}\rho^{(n-1)(1-\frac{p}{2})} d\rho,
\end{align*} where in the last inequality we used \eqref{lower bound for the radon transf of u}.

Let us remark that if $\rho\in (0, A(t)-R-1)$, then for a suitably large constant $C_R$ we have
\begin{align*}
A(t)-\rho+R\leq C_R (A(t)-\rho-R)  \qquad \mbox{and} \qquad A(t)-\rho+\phi(1)\leq C_R (A(t)-\rho-R) .
\end{align*}

Consequently,
\begin{align*}
&\int_{\mathbb{R}^n}|u(t,x)|^p dx \gtrsim (1+t)^{-\frac{\ell}{2} p}\int_0^{A(t)-R-1} \frac{\rho^{(n-1)(1-\frac{p}{2})} }{(A(t)-\rho-R)^{\frac{n-1}{2}\, p+\gamma p-\big(\frac{k+1}{\ell+1}-\gamma p -(n-1)(\frac{p}{2}-1)\big)p}} d\rho.
\end{align*}
However, using the condition $p=p_0(n;\ell,k)$ we get that the exponent of the function at the denominator is exactly $1$.

Hence, for large $t$ it follows
\begin{align*}
\int_{\mathbb{R}^n}|u(t,x)|^p dx &\gtrsim (1+t)^{-\frac{\ell}{2} p}\int_{\frac{A(t)-R-1 }{2}}^{A(t)-R-1} \frac{\rho^{(n-1)(1-\frac{p}{2})} }{A(t)-\rho-R} d\rho\\ &\gtrsim (1+t)^{-\frac{\ell}{2} p}(A(t)-R-1)^{(n-1)(1-\frac{p}{2})} \int_{\frac{A(t)-R-1 }{2}}^{A(t)-R-1} \frac{d\rho}{A(t)-\rho-R} \\  &\gtrsim (1+t)^{-\frac{\ell}{2} p}(A(t)-R-1)^{(n-1)(1-\frac{p}{2})}\log\big(\tfrac{A(t)-R+1}{2}\big) \\
& \gtrsim (1+t)^{-\frac{\ell}{2}p+(n-1)(1-\frac{p}{2})(\ell+1)}\log (1+t).
\end{align*}
Let us point out that in the case $p\geq 2$ in the first line of the previous chain of inequalities we have instead of $\rho^{(n-1)(1-\frac{p}{2})}$ the term $(A(t)+R)^{(n-1)(1-\frac{p}{2})}$. Nevertheless, the final estimate is the same.

Thus,
\begin{align*}
\frac{d^2 \mathscr{G}}{dt^2}(t)&=\int_{\mathbb{R}^n} u_{tt}(t,x)dx= (\ell+1)^2(1+t)^k\int_{\mathbb{R}^n}|u(t,x)|^p dx  \gtrsim (1+t)^{k-\frac{\ell}{2}p+(n-1)(\ell+1)-\frac{n-1}{2}p(\ell+1)} \log(1+t) ,
\end{align*} and  for large $t$ a first integration leads to
\begin{align*}
\frac{d \mathscr{G}}{dt}(t) 
&\gtrsim \frac{d \mathscr{G}}{dt}(0)+(1+t)^{k+1-\frac{\ell}{2}p+(n-1)(\ell+1)-\frac{n-1}{2}p(\ell+1)} \log(1+t).
\end{align*}

Similarly, performing a second integration, for large $t$  we arrive at
\begin{align*}
\mathscr{G}(t)&\gtrsim \mathscr{G}(0)+\frac{d \mathscr{G}}{dt}(0)\, t+(1+t)^{k+2-\frac{\ell}{2}p+(n-1)(\ell+1)-\frac{n-1}{2}p(\ell+1)} \log(1+t)\\ &\gtrsim(1+t)^{k+2-\frac{\ell}{2}p+(n-1)(\ell+1)-\frac{n-1}{2}p(\ell+1)} \log(1+t),
\end{align*} where in the last inequality we used the condition $p=p_{\NE}(n;\ell,k)>p_1(n;\ell,k) $ which implies that the exponent of $(1+t)$ is greater than $1$. But this last estimate is precisely the improvement of \eqref{condition on G Kato} we were looking for. This concludes the proof of Theorem \ref{Main theo} in the critical case.\end{proof}

\section{Proof of Corollary \ref{Blow up cor Tric eq}} \label{Section corollaries for the scale inv case}
\setcounter{equation}{0}

Let us consider the transformation \eqref{transformation u,v}. Choosing  $\ell$ and $k$ as in \eqref{definition ell and k},
the conditions \eqref{1st condition p blow up Tric} and \eqref{2nd condition p blow up Tric} on $p$ may be rewritten as
\begin{align*}
(n-1+\mu_1) p^2- (n+1+\mu_1)p-2<0 \,\,\, \mbox{and} \,\,\, p<p_{\Fuj}\big(n+\tfrac{\mu_1-1}{2}-\tfrac{\sqrt{\delta}}{2}\big) ,
\end{align*}
respectively. We are able to consider for \eqref{CP semi scale inv} the critical cases, too, with the only exception $p=p_0(n+\mu_1)$ in spatial dimension $n=1$. The sign conditions on data are due to
\begin{align*}
u_0(x)&=v_0\big(\tfrac{x}{\sqrt{\delta}}\big),\\
u_1(x)&= \tfrac{1}{\sqrt{\delta}}\big(v_1\big(\tfrac{x}{\sqrt{\delta}}\big)+\big(\tfrac{\mu_1-1+\sqrt{\delta}}{2}\big)v_0\big(\tfrac{x}{\sqrt{\delta}}\big)\big),
\end{align*} where $u_0,u_1$ denote the initial data for the transformed problem \eqref{CP semi Tric eq}.
Therefore, the result follows immediately from Theorem \ref{Main theo}.

\section{Concluding remarks and open problems} \label{Sectionconcludingremarks}
\setcounter{equation}{0}

Let us compare the statements from Theorem \ref{Main theo} with those we have recalled in the introduction.

In the limit case $\delta=1$ Corollary \ref{Blow up cor Tric eq}  coincides with
Theorem 2.6 in \cite{NunPalRei16}. Moreover, this corollary is consistent with Theorem 2.4 in \cite{NunPalRei16} and Theorem 2.4 in \cite{Pal17}.

Corollary \ref{cor damping 1} tells us that for $\mu_1\in [0,1)$ the damping term in the model \eqref{CP semi scale inv mu2=0} is non-effective or hyperbolic-like, being the upper bound for the nonexistence result $p_0(n+\mu_1)$. In particular, for $\mu_1=0$ it coincides with the Strauss exponent appearing in the classical blow-up result for the semi-linear free wave equation with power nonlinearity. Additionally, if we compare the result of this corollary with that one of \cite[Theorem 1.4]{Waka14A}, then we see that our result improves the range for $p$. A heuristic explanation of this fact can be done underlying that our result is proved through Kato's lemma, which is in some sense \textquotedblleft sharper$\!\!$ \textquotedblright $\,$ for hyperbolic-like models.

In the same way Corollary \ref{cor damping 2} says that for $\mu_1\in (1,2]$ the damping term in the model \eqref{CP semi scale inv mu2=0} is hyperbolic-like for spatial dimension $n\geq 2$. Indeed, it results
\begin{align}\label{threshould non-effectiveness}
p_0(n+\mu_1)\geq p_{\Fuj}(n) \qquad \mbox{iff} \qquad 0\leq \mu_1\leq \tfrac{n^2+n+2}{n+2}=:\widetilde{\mu}_1(n).
\end{align} So, in spatial dimension $n=1$ it is appropriate to say that the model is hyperbolic-like only for $\mu_1\in (1,\frac{4}{3}]$. In particular, we see that in the special case $\mu_1=2$ we find the same range for $p$  as in \cite[Theorem 1]{DabbLucRei15}.
Finally, we can compare Corollary \ref{cor damping 2} with Theorem 1.3 in \cite{Waka14A}. Also in this case, although the assumptions on initial data are somehow related, with our approach we are able to improve the range for $p$. We can repeat the same heuristic explanation made in the previous case, since in \cite{Waka14A} the author employs the test function method.

Because of \eqref{definition ell and k} we cannot say anything about the case $\mu_1=1$. Nevertheless, since
\begin{align*}
\lim_{\mu_1\to 1^-}p_{\mu_1}(n)=\lim_{\mu_1\to 1^+}p_{\mu_1}(n)=p_0(n+1),
\end{align*} it would be reasonable to conjecture that
\begin{align*}
p_1(n)=\max\big\{p_0(n+1),p_{\Fuj}(n)\}=p_0(n+1).
\end{align*} As we have already mentioned, the obstacle that does not allow to proceed with our approach the case $\mu_1=1$ is the impossibility to perform the change of variables that we have used in \eqref{transformation u,v}. But probably a different change of variables can be employed in order to study this specific case. More precisely, performing the change of variables
\begin{align*}
1+\tau=e^t, \quad x=y
\end{align*} then $v$ is a solution to \eqref{CP semi scale inv mu2=0} if and only if $v$ solves the following semi-linear Cauchy problem for the wave equation in the \emph{anti de Sitter space-time} \begin{align*}
\begin{cases}
v_{tt}-e^{2t}\Delta v =e^{2t}|v|^p, \qquad t>0, \; x\in \mathbb{R}^n, \\  v(0,x)=v_0(x), \quad x\in \mathbb{R}^n,\\  v_t(0,x)=v_1(x),\quad x\in \mathbb{R}^n.
\end{cases}
\end{align*}

By Corollaries \ref{cor damping 1} and \ref{cor damping 2} it arises quite naturally the question whether also for $\mu_1>2$ in spatial dimension $n\geq 3$ we can expect a hyperbolic-like damping term. According to \eqref{threshould non-effectiveness} this kind of conjecture would be meaningful only for $\mu_1\in (2,\widetilde{\mu}_1(n))$ if $n\geq 3$.

The validity of these two conjectures (namely, the upper bound for $p$ in the case $\mu_1=1$ and the hyperbolic-like damping term for $\mu_1\in (2,\widetilde{\mu}_1(n))$ in spatial dimension $n\geq 3$) has already been proved in \cite{IkedaSob17}. In particular, for a subset of admissible $\mu_1$s our result is cosistent with that in \cite{IkedaSob17}, which improves in turn that one in \cite{LaiTakWak17} and is coherent with the case $\mu_1=2$, for which the value of the critical exponent is known, at least in spatial dimension $n=1,2,3$ and $n\geq 5$ odd (cf. \cite{DabbLucRei15, DabbLuc15}).

We can summarize all results we proved in the present article for the semi-linear wave equation with scale-invariant damping and power non-linearity in  Figure \ref{fig 2x2} (there the thicker lines enclose the range for which we proved a blow up result in the present paper).


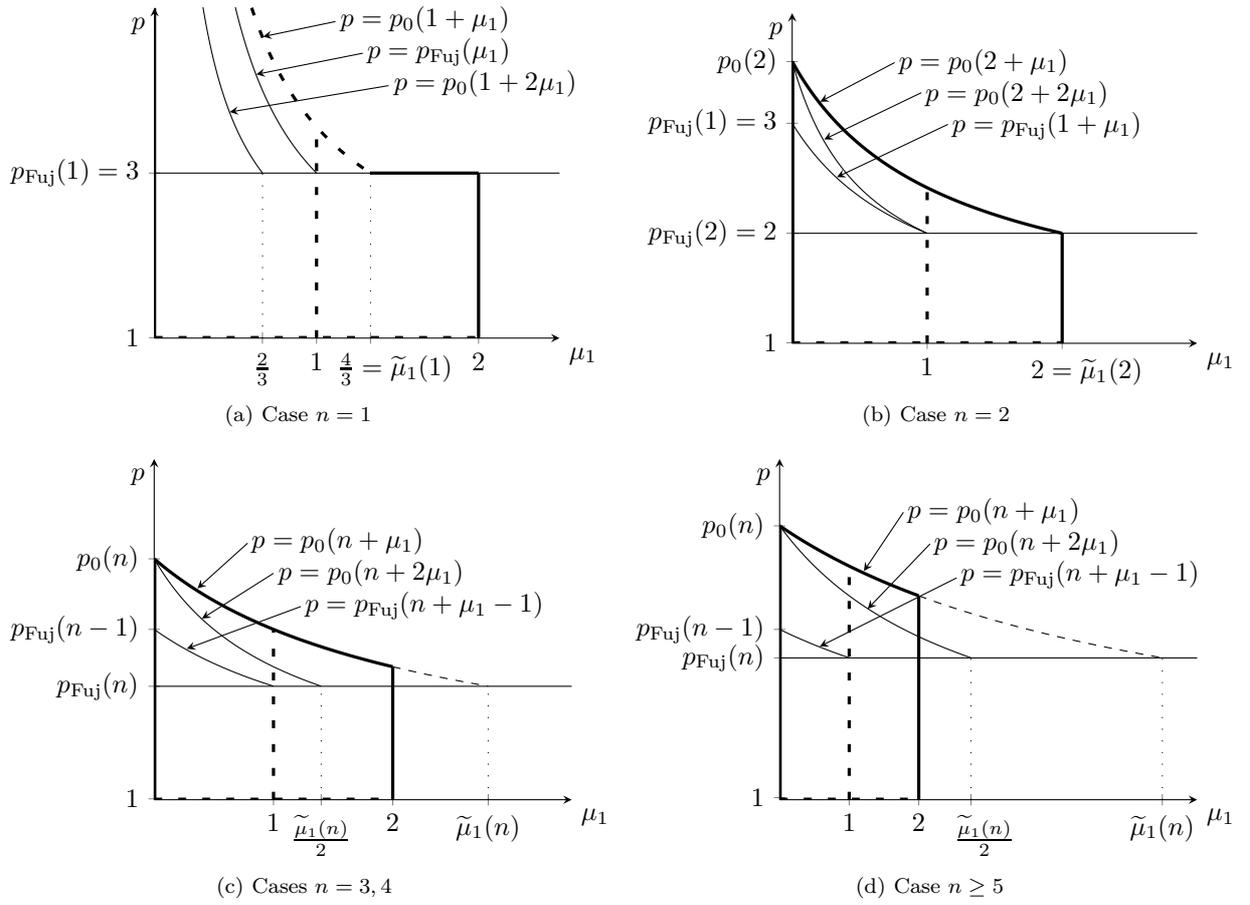
\begin{figure}
	\centering
		\subfloat[Case $n=1$]{\label{figura n=1}
		\begin{tikzpicture}
		
			\begin{axis}
			[width=0.42\textwidth,
			axis x line=middle,
			axis y line=middle,
			xlabel=$\mu_1$,
			ylabel=$p$,
			xlabel style={below right},
			ylabel style={below left},
			xmin=0, xmax=2.5, ymin=1, ymax=5,
			xtick={0.666,1,1.333,2},
			ytick={1.005,3},
			xticklabels={$\frac{2}{3}$, $1$,
			$\!\qquad\frac{4}{3}=\widetilde{\mu}_1(1)$, $2$},
			 yticklabels={$1$, $p_{\Fuj}(1)=3$}]
			\addplot [loosely dashed]
			[domain=0:4/3, samples=50, smooth, very thick]
			{(2+x+((x+1)^2+10*(x+1)-7)^(1/2))/(2*x)};
			\addplot
			[domain=0:2/3, samples=50, smooth, thin]
			{(2+2*x+((2*x+1)^2+10*(2*x+1)-7)^(1/2))/(4*x)};
			\addplot
			[domain=0:1, samples=50, smooth, thin]
			{1+2/x};
			\addplot
			[domain=0:1.333, samples=50, smooth, thin]
			{3};
			\addplot
			[domain=1.33:2, samples=50, smooth, very thick]
			{3};
			\addplot
			[domain=2:2.5, samples=50, smooth, thin]
			{3};
			\addplot [loosely dotted]
			[domain=1:3, variable=\t, samples=50, smooth, thin]
			({2/3},{t});
			\addplot [loosely dashed]
			[domain=1:3.561, variable=\t, samples=50, smooth, very thick]
			({1},{t});
			\addplot [loosely dotted]
			[domain=1:3, variable=\t, samples=50, smooth, thin]
			({4/3},{t});
			\addplot
			[domain=1:3, variable=\t, samples=50, smooth, very thick]
			({2},{t});
			\addplot
			[domain=1:5, variable=\t, samples=50, smooth, very thick]
			({0},{t});
			\addplot [loosely dashed]
			[domain=0:2, samples=50, smooth, ultra thick]
			{1};
			\end{axis}
			\node [align=left] at (4,3.8)
					  { $p=p_0(1+\mu_1)$ \\ $\quad p=p_{\Fuj}(\mu_1)$  \\
					  $\qquad p=p_0(1+2\mu_1)$};
			\draw [stealth - ] (1.43,4) - - (2.38,4.2);
			\draw [stealth - ] (1.34,3.5) - - (2.74,3.78);
			\draw [stealth - ] (0.99,3) - - (3.08,3.35);
			\end{tikzpicture}}
			\hfill
			\subfloat[Case $n=2$]{\label{figura n=2}
			\begin{tikzpicture}
			\begin{axis}
			[width=0.42\textwidth,
			axis x line=middle,
			axis y line=middle,
			xlabel=$\mu_1$,
			ylabel=$p$,
			xlabel style={below right},
			ylabel style={below left},
			xmin=0, xmax=3, ymin=1, ymax=4,
			xtick={1, 2},
			ytick={1.005,2,3,3.561},
			xticklabels={ $1$, $\!\qquad 2=\widetilde{\mu}_1(2)$},
			 yticklabels={$1$, $p_{\Fuj}(2)=2$, $p_{\Fuj}(1)=3$, $p_0(2)$}]
			\addplot
			[domain=0:2, samples=50, smooth, very thick]
			{(3+x+((x+2)^2+10*(x+2)-7)^(1/2))/(2*x+2)};
			\addplot
			[domain=0:1, samples=50, smooth, thin]
			{(3+2*x+((2*x+2)^2+10*(2*x+2)-7)^(1/2))/(4*x+2)};
			\addplot
			[domain=0:1, samples=50, smooth, thin]
			{1+2/(1+x)};
			\addplot
			[domain=0:3, samples=50, smooth, thin]
			{2};
			\addplot [loosely dashed]
			[domain=1:2.414, variable=\t, samples=50, smooth, very thick]
			({1},{t});
			\addplot
			[domain=1:2, variable=\t, samples=50, smooth, very thick]
			({2},{t});
			\addplot
			[domain=1:3.561, variable=\t, samples=50, smooth, ultra thick]
			({0},{t});
			\addplot [loosely dashed]
			[domain=0:2, samples=50, smooth, ultra thick]
			{1};
			\end{axis}
			\node [align=left] at (3,3.3)
					  {  $p=p_0(2+\mu_1)$\\ $\quad p=p_0(2+2\mu_1)$  \\
					   $\qquad p=p_{\Fuj}(1+\mu_1)$ };
			\draw [stealth - ] (0.370,3.2) - - (1.35,3.72);
			\draw [stealth - ] (0.422,2.7) - - (1.7,3.32);
			\draw [stealth - ] (0.589,2.2) - - (2.05,2.9);
		\end{tikzpicture}}
		
		\subfloat[Cases $n=3,4$]{\label{figura n=3,4}
		\begin{tikzpicture}
			\begin{axis}
			[width=0.43\textwidth,
			axis x line=middle,
			axis y line=middle,
			xlabel=$\mu_1$,
			ylabel=$p$,
			xlabel style={below right},
			ylabel style={below left},
			xmin=0, xmax=3.5, ymin=1, ymax=3,
			xtick={1, 1.4, 2, 2.8},
			ytick={1.005,1.666, 2, 2.414},
			xticklabels={ $1$, $\frac{\widetilde{\mu}_1(n)}{2}$,
			$2$, $\widetilde{\mu}_1(n)$},
			 yticklabels={$1$, $p_{\Fuj}(n)$, $p_{\Fuj}(n-1)$, $p_0(n)$}]
			\addplot
			[domain=0:2, samples=50, smooth, very thick]
			{(4+x+((x+3)^2+10*(x+3)-7)^(1/2))/(2*x+4)};
			\addplot [dashed]
			[domain=2:2.8, samples=50, smooth, thin]
			{(4+x+((x+3)^2+10*(x+3)-7)^(1/2))/(2*x+4)};
			\addplot
			[domain=0:1.4, samples=50, smooth, thin]
			{(4+2*x+((2*x+3)^2+10*(2*x+3)-7)^(1/2))/(4*x+4)};
			\addplot
			[domain=0:1, samples=50, smooth, thin]
			{1+2/(2+x)};
			\addplot
			[domain=0:3.5, samples=50, smooth, thin]
			{1.666};
			\addplot [loosely dashed]
			[domain=1:2, variable=\t, samples=50, smooth, very thick]
			({1},{t});
			\addplot [loosely dotted]
			[domain=1:1.666, variable=\t, samples=50, smooth, thin]
			({1.4},{t});
			\addplot [loosely dotted]
			[domain=1:1.666, variable=\t, samples=50, smooth, thin]
			({2.8},{t});
			\addplot
			[domain=1:1.780, variable=\t, samples=50, smooth, very thick]
			({2},{t});
			\addplot [loosely dashed]
			[domain=0:2, samples=50, smooth, ultra thick]
			{1};
			\addplot
			[domain=1:2.414, variable=\t, samples=50, smooth, ultra thick]
			({0},{t});
			\end{axis}
			\node [align=left] at (3.2,3)
					  {  $p=p_0(n+\mu_1)$\\ $\quad p=p_0(n+2\mu_1)$  \\
					   $\qquad p=p_{\Fuj}(n+\mu_1-1)$ };
			\draw [stealth - ] (0.57,2.8) - - (1.23,3.40);
			\draw [stealth - ] (0.63,2.4) - - (1.57,3.00);
			\draw [stealth - ] (0.43,2.0) - - (1.93,2.56);
		\end{tikzpicture}}
		\hfill
		\subfloat[Case $n\geq 5$]{\label{figura n>4}
		\begin{tikzpicture}
			\begin{axis}
			[width=0.43\textwidth,
			axis x line=middle,
			axis y line=middle,
			xlabel=$\mu_1$,
			ylabel=$p$,
			xlabel style={below right},
			ylabel style={below left},
			xmin=0, xmax=6, ymin=1, ymax=1.8,
			xtick={1, 2, 2.75, 5.5}, 
			ytick={1.005,1.333, 1.4, 1.643}, 
			xticklabels={ $1$, $2$, $\quad\frac{\widetilde{\mu}_1(n)}{2}$,
			 $\widetilde{\mu}_1(n)$},
			 yticklabels={$1$, $p_{\Fuj}(n)$, $p_{\Fuj}(n-1)$, $p_0(n)$}]
			\addplot
			[domain=0:2, samples=50, smooth, very thick]
			{(7+x+((x+6)^2+10*(x+6)-7)^(1/2))/(2*x+10)};
			\addplot [dashed]
			[domain=2:5.5, samples=50, smooth, thin]
			{(7+x+((x+6)^2+10*(x+6)-7)^(1/2))/(2*x+10)};
			\addplot
			[domain=0:2.75, samples=50, smooth, thin]
			{(7+2*x+((2*x+6)^2+10*(2*x+6)-7)^(1/2))/(4*x+10)};
			\addplot
			[domain=0:1, samples=50, smooth, thin]
			{1+2/(5+x)};
			\addplot
			[domain=0:6, samples=50, smooth, thin]
			{1.333};
			\addplot [loosely dashed]
			[domain=1:1.548, variable=\t, samples=50, smooth, very thick]
			({1},{t});
			\addplot [loosely dotted]
			[domain=1:1.333, variable=\t, samples=50, smooth, thin]
			({2.75},{t});
			\addplot [loosely dotted]
			[domain=1:1.333, variable=\t, samples=50, smooth, thin]
			({5.5},{t});
			\addplot
			[domain=1:1.479, variable=\t, samples=50, smooth, very thick]
			({2},{t});
			\addplot
			[domain=1:1.643, variable=\t, samples=50, smooth, ultra thick]
			({0},{t});
			\addplot [loosely dashed]
			[domain=0:2, samples=50, smooth, ultra thick]
			{1};
			\end{axis}
			 \node [align=left] at (3.6,3.4)
					  {  $p=p_0(n+\mu_1)$\\ $\quad p=p_0(n+2\mu_1)$  \\
					   $\qquad p=p_{\Fuj}(n+\mu_1-1)$ };
			\draw [stealth - ] (1.11,3.03) - - (1.63,3.81);
			\draw [stealth - ] (1.16,2.53) - - (1.98,3.39);
			\draw [stealth - ] (0.53,2.03) - - (2.33,2.94);
		\end{tikzpicture}}
	\captionsetup{format=hang,labelfont={sf,bf}}
	\caption{Range of the blow-up result for the semi-linear wave equation with  scale-invariant damping}
	\label{fig 2x2}
\end{figure}

Finally, in order to prove that $p_{\mu_1,\mu_2}(n)$ (and, in particular, $p_{\mu_1}(n)$ in the case $\mu_2=0$) is really the critical exponent for the Cauchy problem \eqref{CP semi scale inv} provided that $\delta\in (0,1]$, it remains to prove for an admissible range of parameters $p>p_{\mu_1,\mu_2}(n)$ a global (in time) existence result of small data solutions to \eqref{CP semi scale inv}. For example, something can be done following the approach of \cite[Theorem 1.2]{HeWittYin17}. But, there is a gap between the threshold over which one can prove a global (in time) existence result using the techniques of the above cited paper and $p_{\mu_1,\mu_2}(n)$ (cf. \cite[Chapter 7]{PalThe}). The authors of the above cited paper studied also the super-critical case in \cite{HeWittYinA} in order to fill this gap. However, in order to apply the same approach as in \cite{Geo97} they are forced to modify the nonlinearity.


\appendix

\renewcommand*{\thesection}{\Alph{section}}

\section{Gauss hypergeometric function} \label{Section Gauss hypergeo funct}

For the content of this section we refer to \citep{BE53,AS84}.

The Gauss hypergeometric function is defined for $|z|<1$ by the power series
\begin{align*}
F(a,b,c;z)= \phantom{}_{2}F_1(a,b;c;z)=\sum_{h=0}^\infty \frac{(a)_h (b)_h}{(c)_h} \frac{z^h}{h!}   \qquad \mbox{with} \,\, c\neq 0,-1,-2,\cdots,
\end{align*} where $(q)_n$ denotes the rising factorial.

The hypergeometric function is a solution of Euler's hypergeometric differential equation
\begin{align*}
z(1-z)w^{\prime\prime}(z)+[c-(a+b+1)z]w^\prime (z)-ab w(z)=0.
\end{align*}

When the parameters $(a,b;c)$ are all positive or when $a=b$ and $c>0$, then $F(a,b;c;\cdot)$ is increasing and positive on $[0,1)$. 
Finally,
\begin{align}
F'(a,b;c;z)=\frac{ab}{c} F(a+1,b+1;c+1;z). \label{derivative hypergeometric function}
\end{align}

\section{Hardy-Littlewood maximal function} \label{Section maximal function}

Let $u\in L^1_{\loc}(\mathbb{R}^n)$. The Hardy-Littlewood maximal function of $u$ is
$$M (u )(x)=\sup_{r>0}\,\frac{1}{\meas (B_r(x)) }\int_{B_r(x)}|u(y)|dy,$$
where $\meas (B_r(x))$ denotes the measure of the ball around $x$ with radius $r$.

\begin{prop}[Hardy-Littlewood maximal inequality]\label{L^p boundedness of M operator} Let $u\in L^p(\mathbb{R}^n)$, $1\leq p\leq \infty$. Then $M(u)$ is finite almost everywhere and there exists a constant $C$ depending only on $p$ and $n$ such that the following statements hold:
\begin{enumerate}
\item[\rm{(i)}] if $p=1$, then $$\meas\big(\{x\in\mathbb{R}^n: M(u)(x)>\lambda\}\big)\leq \frac{C}{\lambda}\|u\|_{L^1(\mathbb{R}^n)} \quad \mbox{for any} \,\, \lambda >0;$$
\item[\rm{(ii)}] if $1<p\leq \infty$, then $$\|M(u)\|_{L^p(\mathbb{R}^n)}\leq C\|u\|_{L^p(\mathbb{R}^n)}.$$
\end{enumerate}
\end{prop}

For the proof of the previous result one can see \cite[Theorem 1 at page 13]{Stein93}.

\section{Weak $L^p$ spaces} \label{Section weak Lp}

In this section we recall the definition of the spaces $L^{p,\infty}(X)$ and we collect the main results that are employed in Section  \ref{Section repres formula Tric eq}, where throughout this section $(X,\mathfrak{M},\mu)$ is a measure space and $0<p\leq \infty$. We follow here the treatment of \cite{Gra04}.

For a given measurable function $f$ on $X$ the \emph{distribution function} of $f$ is the function $d_f$ which is defined as follows:
\begin{align*}
d_f(\alpha)=\mu\big(\{x\in X: |f(x)|>\alpha\}\big) \qquad \mbox{for any} \,\,\, \alpha > 0.
\end{align*}


 For $0<p<\infty$, the space \emph{weak} $L^p(X,\mu)$, denoted by $L^{p,\infty}(X,\mu)$, is defined as the set of all $\mu$-measurable functions $f$ such that
\begin{align*}
\| f\|_{L^{p,\infty}(X)}&=\inf \{C>0: d_f(\alpha)\leq \tfrac{C^p}{\alpha^p} \qquad \mbox{for any} \,\,\, \alpha>0\} = \sup_{\alpha>0}\,\,\alpha \, d_f(\alpha)^{\frac{1}{p}} <\infty.
\end{align*} The space \emph{weak} $L^\infty(X,\mu)$ is by definition $L^\infty(X,\mu)$.

The space $L^{p,\infty}(X,\mu)$ is a quasinormed linear space for $0<p<\infty$. The weak $L^p$ spaces are larger than the usual $L^p$ spaces. Indeed, it holds the following result.
\begin{prop} Let $0<p<\infty$. Then $L^{p}(X,\mu)\subset L^{p,\infty}(X,\mu)$ and in particular for any $f\in L^p(X,\mu)$ we have $$\|f\|_{L^{p,\infty}(X)}\leq \|f\|_{L^{p}(X)}.$$
\end{prop} Moreover, the inclusion $L^{p}(X,\mu)\subset L^{p,\infty}(X,\mu)$ is strict for $0<p<\infty$. For example, we can consider the function $h(x)=|x|^{-\frac{n}{p}}$ for $x\in \mathbb{R}^n$. Clearly, $h$ is not in $L^p(\mathbb{R}^n)$, but $h$ is in $L^{p,\infty}(\mathbb{R}^n)$ with $\|h\|_{L^{p,\infty}(\mathbb{R}^n)}^p=\nu_n$, where $\nu_n$ is the $n$-dimensional Lebesgue measure of the unit ball of $\mathbb{R}^n$.

The weak $L^p$ spaces are complete with respect to the quasinorm $\|\cdot\|_{L^{p,\infty}(X)}$. Furthermore, $L^{p,\infty}(X,\mu)$ is metrizable  for $0<p<\infty$ and normable for $p>1$. Hence, $L^{p,\infty}(X,\mu)$ is a quasi-Banach space  for $0<p<\infty$ (Banach space for $p>1$, respectively).

\begin{rmk}
Let $f$, $g$ be measurable function such that $|f|\leq |g|$ $\mu$-a.e.. Then, $\|f\|_{L^{p,\infty}(X)}\leq \|g\|_{L^{p,\infty}(X)} $.
\end{rmk}

 The next result corresponds to the weak $L^p$ spaces version of H\"{o}lder's inequality  (cf. \cite[Exercise 1.1.15]{Gra04}).
\begin{prop}[H\"{o}lder's inequality for weak type spaces]\label{Holder ineq weak Lp}
Let $(X,\mu)$ be a measurable space and let $\{f_j\}_{1\leq j\leq k}$ be measurable functions on $X$.
 Let us assume $f_j\in L^{p_j,\infty}(X)$ with $0<p_j<\infty$ for any $1\leq j\leq k$. Define $0<p< \infty$ such that $$\frac{1}{p}=\sum_{j=1}^k \frac{1}{p_j}.$$ Then $\prod_{j=1}^k f_j\in L^{p,\infty}(X)$ and, furthermore,
\begin{align*}
\Big\|\prod_{j=1}^k f_j\Big\|_{L^{p,\infty}(X)}\leq p^{-\frac{1}{p}}\prod_{j=1}^k p_j^{\frac{1}{p_j}}\prod_{j=1}^k \| f_j\|_{L^{p_j,\infty}(X)}.
\end{align*}
\end{prop}

 Finally, we state two results for weak $L^p$ spaces that we use in Section \ref{Section repres formula Tric eq} together with Proposition \ref{Holder ineq weak Lp}. The first result is the counterpart of Young's inequality for weak spaces. 

\begin{prop}[Young's inequality for weak type spaces]\label{Young ineq weak spac} Let $1\leq p <\infty$ and $1<q,r<\infty$ satisfy $$\frac{1}{r}+1=\frac{1}{p}+\frac{1}{q}.$$ Then there exists a constant $C=C(p,q,r)>0$ such that for all $f\in L^p(\mathbb{R}^n)$ and $g\in L^{q,\infty}(\mathbb{R}^n)$ we have
\begin{align*}
\|f\ast g\|_{L^{r,\infty}(\mathbb{R}^n)}\leq C\,  \|g\|_{L^{q,\infty}(\mathbb{R}^n)}\|f\|_{L^{p}(\mathbb{R}^n)}.
\end{align*}
\end{prop}

The next interpolation result can be stated in the general frame of measurable spaces.

\begin{prop}[Marcinkiewicz interpolation theorem]\label{Marcinkiewicz inter thrm} Let $(X,\mu)$ and $(Y,\nu)$ be measurable spaces and let $0<p_0<p_1\leq \infty$. Let $T$ be a linear operator defined on the space $L^{p_0}(X)+L^{p_1}(X)$ and taking values in the space of measurable functions on $Y$. Assume that there exists two positive constant $A_0$ and $A_1$ such that
\begin{align*}
&\|T f\|_{L^{p_0,\infty}(Y)}\leq A_0 \|f\|_{L^{p_0}(X)} \qquad \mbox{for any} \,\,\, f\in L^{p_0}(X) ,\\&\|T f\|_{L^{p_1,\infty}(Y)}\leq A_1 \|f\|_{L^{p_1}(X)} \qquad \mbox{for any} \,\,\, f\in L^{p_1}(X).
\end{align*} Then for all $p\in (p_0,p_1)$ and for all $f\in L^p(X)$ we have the estimate
\begin{align*}
\|T f\|_{L^{p}(Y)}\leq A \|f\|_{L^{p}(X)},
\end{align*} where
\begin{align*}
A=2\Big(\tfrac{p}{p-p_0}+\tfrac{p}{p-p_1}\Big)^{\frac{1}{p}}A_0^{1-\theta}A_1^{\theta}, \qquad \mbox{with} \,\,\, \theta \in (0,1) \,\,\, \mbox{determined through} \,\,\,  \tfrac 1p=\tfrac{1-\theta}{p_0}+\tfrac{\theta}{ p_1}.
\end{align*}
\end{prop} For the proofs of the previous propositions one can see \cite[Theorem 1.2.13 and Theorem 1.3.2]{Gra04}.

\section*{Acknowledgments}

The PhD study of the first author is supported by S\"{a}chsiches Landesgraduiertenstipendium. The first author is member of the Gruppo Nazionale per L'Analisi Matematica, la Probabilit\`{a} e le loro Applicazioni (GNAMPA) of the Instituto Nazionale di Alta Matematica (INdAM). Both authors thank Karen Yagdjian (Edinburg, Texas) for his useful comments to prepare the final version of this paper.




\section*{References}

\begin{thebibliography}{00}



\bibitem{AS84} M. Abramowitz and I. A. Stegun, editors.
\newblock {\em Handbook of mathematical functions.}
\newblock { Verlag Harri Deutsch, Thun,} 1984.


\bibitem{BE53} H. Bateman and A. Erdelyi.
\newblock {\em Higher transcendental functions, Vol.1.}
\newblock { McGraw-Hill, New York,} 1953.



\bibitem{DabbLucMod} M. D'Abbicco, S. Lucente.
\newblock {A modified test function method for damped wave equation.}
\newblock {\em  Advanced Nonlinear Studies} {\bf 13} (2013), 863-889.

\bibitem{DabbLuc15} M. D'Abbicco, S. Lucente.
\newblock{ NLWE with a special scale invariant damping in odd space dimension.}
\newblock{\em Discrete Contin. Dyn. Syst.} 2015,
\newblock{ Dynamical systems, differential equations and applications.} 10th AIMS Conference. Suppl., 312-319.


\bibitem{DabbLucRei13} M. D'Abbicco, S. Lucente, M. Reissig.
\newblock {Semi-linear wave equations with effective damping.}
\newblock {\em Chinese Annals of Mathematics, Ser. B, } {\bf 34} (2013), 345-380.

\bibitem{DabbLucRei15} M. D'Abbicco, S. Lucente,  M. Reissig.
\newblock { A shift in the Strauss exponent for semi-linear wave equations with a not effective damping.}
\newblock {\em J. Diff. Equations } {\bf 259} (2015), 5040-5073.



\bibitem{Geo97} V. Georgiev, H. Lindblad, C. D. Sogge.
\newblock {Weighted Strichartz estimates and global existence for semi-linear wave equations.}
\newblock {\em Amer. J. Math.} {\bf 119} (1997), no. 6, 1291-1319.


\bibitem{Glas81B} R. T. Glassey.
\newblock {Finite-time blow-up for solutions of nonlinear wave equations.}
\newblock {\em Math Z.} {\bf 177} (1981), no. 3, 323-340.

\bibitem{Gra04} L. Grafakos.
\newblock{\em Classical and modern Fourier analysis.}
\newblock{ Prentice Hall,} 2004.


\bibitem{HeWittYin17} D. He, I. Witt, H. Yin.
\newblock{ On the global solution problem for semilinear generalized Tricomi equations, I.}
\newblock{\em Calc. Var.} (2017) doi:10.1007/s00526-017-1125-9.

\bibitem{HeWittYinA} D. He, I. Witt, H. Yin.
\newblock{ On the global solution problem for semilinear generalized Tricomi equations, II.}
\newblock{ arXiv:1611.07606}, (2016).


\bibitem{HeWittYinA17} D. He, I. Witt, H. Yin.
\newblock{ On semilinear Tricomi equations with critical exponents or in
two space dimensions.}
\newblock{ arXiv:1704.07051}, (2017).


\bibitem{HongLi96} J. Hong, G. Li.
\newblock{ $L^p$ estimates for a class of integral operators}
\newblock{\em J. Partial Differ. Equ.} {\bf 9} (1996), 343-364.


\bibitem{IkedaSob17} M. Ikeda, M. Sobajima.
\newblock{Life-span of solutions to semilinear wave equation with time-dependent critical damping for specially localized initial data.} arXiv:1709.04406, (2017).

\bibitem{Jiao03} H. Jiao and Z. Zhou.
\newblock { An elementary proof of the blow-up for semilinear wave equation in high space dimensions.}
\newblock{\em J. Diff. Equations} {\bf 189} (2003), no. 2, 355-365.

\bibitem{John79} F. John.
\newblock {Blow-up of solutions of non-linear wave equations in three space dimensions.}
\newblock{\em Manuscripta Math.} {\bf 28} (1979), no. 1-3, 235-268.

\bibitem{Kato80} T. Kato.
\newblock{ Blow-up of solutions of some nonlinear hyperbolic equations.} \newblock{\em Comm. Pure Appl. Math.} {\bf 33} (1980), no. 4, 501-505.

\bibitem{TAO97}  M. Keel, T. Tao.
\newblock {Small data blow-up for semi-linear Klein-Gordon equations.}
\newblock {\em Amer. J. Math. } {\bf 121} (1997), 629-669.


\bibitem{LaiTakWak17} N.-A. Lai, H. Takamura, K. Wakasa.
\newblock{Blow-up for semilinear wave equations with the scale invariant damping and super-Fujita exponent.}
\newblock{\em J. Diff. Equations} (2017), http://dx.doi.org/10.1016/j.jde.2017.06.017


\bibitem{LinNashZhai12} J. Lin, K. Nishihara, J. Zhai.
\newblock{ Critical exponent for the semilinear wave equation with time-dependent damping}.
\newblock{\em Discrete Contin. Dyn. Syst.} Ser. A {\bf 32} (2012) 4307-4320.



\bibitem{Nishi11} K. Nishihara.
\newblock{Asymptotic behavior of solutions to the semilinear wave equation with time-dependent
damping.}
\newblock{\em Tokyo J. of Math.} {\bf 34} (2011), 327-343.



\bibitem{NunPalRei16} W. Nunes do Nascimento, A. Palmieri, M. Reissig.
\newblock{ Semi-linear wave models with power non-linearity and scale-invariant time-dependent mass and dissipation.}
\newblock{\em Math. Nachr.} (2016) doi:10.1002/mana.201600069.


\bibitem{Pal17}   A. Palmieri.
\newblock{Global existence of solutions for semi-linear wave equation with scale-invariant damping and mass in exponentially weighted spaces.}
submitted (2017).

\bibitem{PalThe} A. Palmieri.
\newblock{PhD thesis}, TU Bergakademie Freiberg, 2017.

\bibitem{PalRei17}   A. Palmieri, M. Reissig.
\newblock{Semi-linear wave models with power non-linearity and
scale-invariant time-dependent mass and dissipation, II.}
submitted (2017).

\bibitem{Scha85} J. Schaeffer.
\newblock {The equation $\square u = |u|^p$ for the critical value of $p$.}
\newblock {\em Proc. Roy. Soc. Edinburgh Sect. A.} {\bf 101} (1985), no. 1-2, 31-44.

\bibitem{Sid84} T. C. Sideris.
\newblock { Non-existence of global solutions to semi-linear wave equations in high dimensions.}
\newblock {\em  J. Diff. Equations} {\bf 52} (1984), 378-406.

\bibitem{Stein93} E. M. Stein.
\newblock{\em Harmonic analysis.}
\newblock{ Princeton University Press, Princeton, New Jersey,} 1993.

\bibitem{Tak15}H. Takamura.
\newblock{ Improved Kato's lemma on ordinary differential inequality and its application to semilinear wave
equations.}
\newblock{\em Nonlinear Anal.} {\bf 125} (2015), 227-240.


\bibitem{TodYor01} G. Todorova, B. Yordanov.
\newblock{Critical exponent for a nonlinear wave equation with damping.}
\newblock{\em J. Diff. Equations } {\bf 174} (2001), 464-489.



\bibitem{Waka14A} Y. Wakasugi.
\newblock{Critical exponent for the semilinear wave equation with scale invariant damping.}
\newblock{\em Fourier Analysis},
\newblock{Trends Math.},\newblock{\em Birkha\"auser/Springer}, \newblock{\em Cham}, (2014), 375-390.


\bibitem{WiThe} J. Wirth.
\newblock{ Asymptotic properties of solutions to wave equations with time-dependent dissipation},
\newblock{ PhD thesis, TU Bergakademie Freiberg,} 2005, 146 pp.


\bibitem{WirthD} J. Wirth.
\newblock { Wave equation with time-dependent dissipation II. Effective dissipation.}
\newblock {\em J. Diff. Equations} {\bf 232} (2007), 74--103.


\bibitem{Yag04} K. Yagdjian.
\newblock{ A note on the fundamental solution for the Tricomi-type equation in the hyperbolic domain.}
\newblock{\em  J. Diff. Equations} {\bf 206} (2004), 227-252.

\bibitem{Yag05} K. Yagdjian.
\newblock{Global existence in Cauchy problem for nonlinear wave equations with variable speed of propagation.}
\newblock{\em Oper. Theory Adv. Appl.,} {\bf 159} \newblock{\em Adv. Partial Differ. Equ. (Basel)} (2005), 301-385.

\bibitem{Yag15} K. Yagdjian.
\newblock{Integral transform approach to generalized Tricomi equations.}
\newblock{\em  J. Differential Equations} {\bf 259} (2015), no. 11, 5927-5981.

\bibitem{Yag16} K. Yagdjian.
\newblock{Integral transform approach to time-dependent partial differential equations.}
\newblock{\em Mathematical analysis, probability and applications--plenary lectures}, 281-336, Springer Proc. Math. Stat., 177, Springer, 2016.


\bibitem{YagGal09} K. Yagdjian, A. Galstian.
\newblock{ Fundamental Solutions for the Klein-Gordon Equation in de Sitter Spacetime.}
\newblock{\em  Comm. Math. Phys.} {\bf 285} (2009), 293-344.

\bibitem{Yor06} B. T. Yordanov, Q. S. Zhang.
\newblock{ Finite time blow up for critical wave equations in high dimensions.}
\newblock{\em J. Func. Anal.} {\bf 231} (2006), 361-374.

\bibitem{Zhang} Q. S. Zhang.
\newblock{ A blow-up result for a nonlinear wave equation with damping: the critical case.}
\newblock{\em C. R. Acad. Sci. Paris Ser.I Math.} {\bf 333} (2001), 109-114.



\bibitem{Zhou07} Y. Zhou.
\newblock{ Blow up of solutions to semilinear wave equations with critical exponent in high dimensions.}
\newblock{\em Chinese Annals of Mathematics, Ser. B,} {\bf 28} (2007), no. 2, 205-212.







\end{thebibliography}


\end{document}